\documentclass[12pt]{article}

\usepackage{amssymb}
\usepackage{graphicx}
\usepackage{epsfig}
\usepackage{here}
\usepackage{psfrag}  
\usepackage{latexsym}
\usepackage{epic}
\usepackage{eepic}
\usepackage{amsmath}
\usepackage{caption}
\usepackage{enumerate}
\usepackage{appendix}
\usepackage[T1]{fontenc}
\usepackage{pstcol}
\usepackage{amsthm}
\usepackage{url}

\usepackage{alltt}

\usepackage[margin=1in]{geometry}

\newcommand{\bea}{\begin{eqnarray}}
\newcommand{\eea}{\end{eqnarray}}
\newcommand{\be}{\begin{equation}}
\newcommand{\ee}{\end{equation}}
\newcommand{\beann}{\begin{eqnarray*}}
\newcommand{\eeann}{\end{eqnarray*}}
\newcommand{\balnn}{\begin{align*}}
\newcommand{\ealnn}{\end{align*}}

\newcommand{\nn}{\nonumber}

\newcommand{\R}{{\mathbb R}}

\newcommand{\ba}{\begin{array}}
\newcommand{\ea}{\end{array}}
\newcommand{\bd}{\begin{displaymath}}
\newcommand{\ed}{\end{displaymath}}
\def\P{\mathbb P}

\newcommand{\1}{{\mathbf 1}}
\def\E{{\mathbb E}}
\newcommand{\F}{{\mathcal F}}

\newcommand{\BB}{{\mathcal B}}

\DeclareMathOperator{\Cov}{Cov}

\DeclareMathOperator{\e}{e}

\newtheorem{thm}{Theorem}

\newtheorem{lemma}{Lemma}

\newtheorem{definition}{Definition}

\title{A $J$-function for inhomogeneous spatio-temporal point processes}
\author{O.\ Cronie\footnote{Corresponding author: ottmar@cwi.nl} \, and  
M.N.M.\ van Lieshout}
\date{}

\begin{document}
\newcommand{\edt}[1]{{\vbox{ \hbox{#1} \vskip-0.3em \hrule}}}

\maketitle \noindent 
\begin{center} 
{\em CWI } \\ 
P.O. Box 94079, 
1090 GB Amsterdam, 
The Netherlands
\end{center} 

\bigskip

\noindent{\bf Abstract}: 
We propose a new summary statistic for inhomogeneous intensity-reweighted 
moment stationary spatio-temporal point processes. The statistic is defined 
through the $n$-point correlation functions of the point process and it 
generalises the $J$-function when stationarity is assumed. We show that our 
statistic can be represented in terms of the generating functional and that 
it is related to the inhomogeneous $K$-function. We further discuss its 
explicit form under some specific model assumptions and derive a 
ratio-unbiased estimator. We finally illustrate the use of our statistic 
on simulated data.

\bigskip


\section{Introduction}

A spatio-temporal point pattern can be described as a collection of pairs
$\{(x_i,t_i)\}_{i=1}^{m}$, $m\geq0$, where $x_i\in W_S\subseteq\R^d$, $d\geq1$, 
and $t_i\in W_T\subseteq\R$ describe, respectively, the spatial location and 
the occurrence time associated with the $i$th event. Examples of such point 
patterns include recordings of earthquakes, disease outbreaks and fires
(see e.g.\ \cite{GabrielDiggle,MollerDiaz,Ogata}). 

When modelling spatio-temporal point patterns, the usual and natural approach 
is to assume that $\{(x_i,t_i)\}_{i=1}^{m}$ constitutes a realisation of a 
spatio-temporal point process (STPP) $Y$ restricted to $W_S\times W_T$. 
Then, in order to deduce what type of model could describe the observations
$\{(x_i,t_i)\}_{i=1}^{m}$, one carries out an exploratory analysis of the data  
under some minimal set of conditions on the underlying point process $Y$. 
At this stage, one is often interested in detecting tendencies for points
to cluster together, or to inhibit one another. In order to do so, one usually
employs spatial or temporal summary statistics, which are able to capture and 
reflect such features. 

A simple and convenient working assumption for the underlying point process 
is stationarity. In the case of a purely spatial point pattern 
$\{x_i\}_{i=1}^{m}\subseteq W_S$ generated by a stationary spatial point 
process $X$, a variety of summary statistics have been developed, see e.g.\
\cite{SKM,Handbook,Illian,MCbook}. One such statistic is the so-called 
$J$-function \cite{MCBaddeley}, given by 
\begin{equation}
\label{e:J}
J(r)= \frac{1-G(r)}{1-F(r)}
\end{equation}
for $r\geq0$ such that $F(r)\neq1$. Here, the empty space function $F(r)$ 
is the probability of having at least one point of $X$ within distance $r$ 
from the origin whereas the nearest neighbour distance distribution function 
$G(r)$ is the conditional probability of some further point of $X$ falling
within distance $r$ from a typical point of $X$. Hence, $J(r)<1$ indicates clustering,
$J(r)=1$ indicates spatial randomness and $J(r)>1$ indicates regularity at 
inter-point distance $r$. 

In many applications, though, stationarity is not a reasonable assumption. 
This observation has led to the development of summary statistics being able to 
compensate for inhomogeneity. For purely spatial point processes,
\cite{BaddeleyEtAl} introduced the notion of second order intensity-reweighted 
stationarity (SIRS) and defined a summary statistic $K_{\rm inhom}(r)$. It can be interpreted as an analogue of the $K$-function, which is proportional to the expected number of further points within distance $r$ of a typical point of $X$, since it reduces to $K(r)$ when $X$ is stationary.

The concept of SIRS was extended to the spatio-temporal case by 
\cite{GabrielDiggle} who also defined an inhomogeneous spatio-temporal 
$K$-function $K_{\rm inhom}(r,t)$, $r,t\geq0$. These ideas were further 
developed and studied in \cite{MollerGhorbani} with particular attention
to the notion of space-time separability.

To take into account interactions of order higher than two, \cite{MCJfun}
introduced the concept of intensity-reweighted moment stationarity (IRMS) 
for purely spatial point processes and generalised (\ref{e:J}) to 
IRMS point processes. 

In this paper we develop a proposal given in \cite{MCJfun} to study the 
spatio-temporal generalisation $J_{\rm inhom}(r,t)$ of (\ref{e:J}) under
suitable intensity-reweighting. In Section~\ref{SectionPreliminaries} we 
give the required preliminaries, which include definitions of product densities, 
Palm measures, generating functionals, $n$-point correlation functions and 
IRMS for spatio-temporal point processes. Then, in Section~\ref{SectionJfun}, 
we give the definition of $J_{\rm inhom}(r,t)$ under the assumption of IRMS and 
discuss its relation to the inhomogeneous spatio-temporal $K$-function of 
\cite{GabrielDiggle}. In Section~\ref{SectionRepresentation} we write 
$J_{\rm inhom}(r,t)$ as a ratio of $1-F_{\rm inhom}(r,t)$ and $1-G_{\rm inhom}(r,t)$ 
in analogy with (\ref{e:J}). As a by-product we obtain generalisations of 
the empty space function and the nearest neighbour distance distribution. The 
section also includes a representation in terms of the Papangelou conditional 
intensity. In Section~\ref{SectionExamples} we consider three classes of 
spatio-temporal point processes for which the IRMS assumption holds, namely 
Poisson processes, location dependent thinning of stationary STPPs and 
log-Gaussian Cox processes. In Section~\ref{SectionEstimation} we derive a 
non-parametric estimator $\widehat{J_{\rm inhom}}(r,t)$ for which we show 
ratio-unbiasedness and in Section~\ref{SectionData} we illustrate its use on 
simulated data.

\section{Definitions and preliminaries}
\label{SectionPreliminaries}

\subsection{Simple spatio-temporal point process}

In order to set the stage, let $\|x\|= (\sum_{i=1}^{d}x_i^2 )^{1/2}$ and 
$d_{\R^d}(x,y)=\|x-y\|$,  $x,y\in\R^d$ denote respectively the Euclidean norm 
and metric. Since space and time must be treated differently, we endow $\R^{d}\times\R$ with the supremum norm 
$\|(x,t)\|_{\infty} = \max\{\|x\|,|t|\}$ and the supremum metric 
\[
d((x,t),(y,s)) = \|(x,t)-(y,s)\|_{\infty} =\max\{d_{\R^d}(x,y),d_{\R}(t,s)\},
\] 
where $(x,t),(y,s)\in\R^{d}\times\R$. Then $(\R^{d}\times\R,  d(\cdot,\cdot))$ 
is a complete separable metric space, which is topologically equivalent to 
the Euclidean space $(\R^{d}\times\R, d_{\R^{d+1}}(\cdot,\cdot))$. Note that 
in the supremum metric a closed ball of radius $r\geq0$ centred at the 
origin $0\in\R^{d}\times\R$ is given by the cylinder set 
\[
B[0,r] = \{ (x,t)\in\R^{d}\times\R: \max\{\|x\|,|t|\}\leq r\}. 
\]

Write $\BB(\R^{d}\times\R) = \BB(\R^{d})\otimes\BB(\R)$ for the $d$-induced 
Borel $\sigma$-algebra and let $\ell$ denote Lebesgue measure on 
$\R^{d}\times\R$. Furthermore, given some Borel set $A\subseteq\R^{d}\times\R$ 
and some measurable function $f$, we interchangeably let $\int_{A}f(y)dy$ 
and $\int_{A}f(y)\ell(dy)$ represent the integral of $f$ over $A$ with 
respect to $\ell$. 

In this paper, a spatio-temporal point process is a simple point process
on the product space $\R^d\times\R$. More formally, let
$N$ be the collection of all locally finite counting measures 
$\varphi$ on $\BB(\R^{d}\times\R)$, i.e.\ $\varphi(A)<\infty$ for 
bounded $A\in\BB(\R^{d}\times\R)$, and let $\mathcal{N}$ be the smallest 
$\sigma$-algebra on $N$ to make the mappings $\varphi\mapsto\varphi(A)$ 
measurable for all $A\in\BB(\R^{d}\times\R)$. Consider in addition the 
sub-collection 
$N^*=\{\varphi\in N: 
  \varphi(\{(x,t)\})\in\{0,1\}\text{ for any }(x,t)\in\R^{d}\times\R\}$ 
of simple elements of $N$.

\begin{definition}
A simple spatio-temporal point process (STPP) $Y$ on $\R^{d}\times\R$ is 
a measurable mapping from some probability space $(\Omega,\F,\P)$ into 
the measurable space $(N,\mathcal{N})$ such that $Y$ almost surely (a.s.)
takes values in $N^*$. 
\end{definition}

Throughout we will denote the $Y$-induced probability measure on 
$\mathcal{N}$ by $P$. To emphasise the counting measure nature of $Y$,
we will sometimes write $Y=\sum_{i=1}^{\infty}\delta_{(X_i,T_i)}$ as a sum of
Dirac measures, where the $X_i \in \R^d$ are the spatial components and the  
$T_i \in \R$ are the temporal components of the points of $Y$. Hence, both 
$Y(\{(x,t)\})=1$ and $(x,t)\in Y$ will have the same meaning and both 
$Y(A)$ and $|Y\cap A|$ may be used as notation for the the number of 
points of $Y$ in some set $A$, where $|\cdot|$ denotes cardinality.

\subsection{Product densities and $n$-point correlation functions}
\label{SectionIRMS}

Our definition of the inhomogeneous $J$-function relies on the so-called
$n$-point correlation functions, which are closely related to the better 
known product densities. Here we recall their definition.

Suppose that the factorial moment measures of $Y$ exist as locally finite
measures and that they are absolutely continuous with respect to the $n$-fold product
of $\ell$ with itself. The Radon--Nikodym derivatives $\rho^{(n)}$, $n\geq1$, referred
to as \emph{product densities}, are permutation invariant and defined by the 
integral equations 
\begin{align}
\label{ExpProdDens}
&\E\left[
\sum_{(x_1,t_1),\ldots,(x_n,t_n)\in Y}^{\neq}
h((x_1,t_1),\ldots,(x_n,t_n))
\right]
=
\\
&=
\int \cdots \int 
h((x_1,t_1),\ldots,(x_n,t_n))
\rho^{(n)}((x_1,t_1),\ldots,(x_n,t_n)) dx_1 dt_1 \cdots dx_n dt_n
\nn
\end{align} 
for non-negative measurable functions $h: (\R^d\times \R)^n \to \R$
under the proviso that the left hand side is infinite if and only if 
the right hand side is. 
Equation (\ref{ExpProdDens}) is sometimes referred to as the {\em Campbell theorem}. 
The heuristic interpretation of 
$\rho^{(n)}((x_1,t_1),\ldots,(x_n,t_n)) dx_1 dt_1 \cdots dx_n dt_n$ is 
that it represents the infinitesimal probability of observing the points 
$x_1,\ldots,x_n\in\R^{d}$ of $Y$ at the respective event times 
$t_1,\ldots,t_n\in\R$.

For $n=1$, we obtain the \emph{intensity measure} $\Lambda$ of $Y$ as
\[
\Lambda(A)
=
\int_{B\times C}
\rho^{(1)}(x,t)dx dt
\]
for any $A = B\times C\in\BB(\R^{d}\times\R)$. 
We shall also use the common notation $\lambda(x,t) = \rho^{(1)}(x,t)$
and assume henceforth that $\bar{\lambda} = \inf_{(x,t)}\lambda(x,t)>0$. 

The  \emph{$n$-point correlation functions} \cite{White} are defined 
in terms of the $\rho^{(n)}$ by setting $\xi_1 \equiv1$ and recursively
defining 
\begin{align}
\label{CorrFunc}
\frac{
\rho^{(n)}((x_1,t_1),\ldots,(x_n,t_n))
}
{\prod_{k=1}^{n}\lambda(x_k,t_k)}
&=
\sum_{k=1}^{n}
\sum_{D_1,\ldots,D_k}
\prod_{j=1}^{k}
\xi_{|D_j|}(\{(x_i,t_i):i\in D_j\}),
\end{align}
where 
$\sum_{D_1,\ldots,D_k}$ is a sum over all possible $k$-sized partitions 
$\{D_1,\ldots,D_k\}$, $D_j\neq \emptyset$, of the set $\{1,\ldots,n\}$ 
and $|D_j|$ denotes the cardinality of $D_j$. 

For a Poisson process on $\R^{d}\times\R$ with intensity function $\lambda(x,t)$,
due to e.g.\ \cite[Theorem 1.3]{MCbook},  we have that 
$\rho^{(n)}((x_1,t_1),\ldots,(x_n,t_n))=\prod_{k=1}^{n}\lambda(x_k,t_k)$, 
whereby $\xi_n\equiv0$ for all $n\geq 2$. Hence, the sum on the right hand 
side in expression (\ref{CorrFunc}) is a finite series expansion of the 
dependence correction factor by which we multiply the product density 
$\prod_{k=1}^{n}\lambda(x_k,t_k)$ of the Poisson process to obtain the 
product density $\rho^{(n)}((x_1,t_1),\ldots,(x_n,t_n))$ of $Y$. 

A further interpretation is obtained by realising that the right hand side 
of the above expression is a series expansion of a higher order version 
of the \emph{pair correlation function}
$$
g((x_1,t_1),(x_2,t_2))=
\frac{\rho^{(2)}((x_1,t_1),(x_2,t_2))}{\lambda(x_1,t_1)\lambda(x_2,t_2)} = 
1 + \xi_2((x_1,t_1),(x_2,t_2)).
$$
 
The main definition of this section gives the class of STPPs to which we
shall restrict ourselves in the sequel of this paper.

\begin{definition}
\label{d:IRMS}
Let $Y$ be a spatio-temporal point process for which product densities 
of all orders exist.
If $\bar{\lambda} = \inf_{(x,t)}\lambda(x,t)>0$ and for all $n\geq1$, 
$\xi_n$ is translation invariant in the sense that 
$$
\xi_n((x_1,t_1)+(a,b),\ldots,(x_n,t_n)+(a,b)) = \xi_n((x_1,t_1),\ldots,(x_n,t_n)) 
$$
for almost all $(x_1,t_1),\ldots,(x_n,t_n)\in\R^{d}\times\R$ and all 
$(a,b)\in\R^{d}\times\R$, we say that $Y$ is intensity-reweighted moment
stationary (IRMS).
\end{definition}

By equation (\ref{CorrFunc}), translation invariance of all $\xi_n$ is 
equivalent to translation invariance of the intensity-reweighted product 
densities. The property is weaker than stationarity (barring the degenerate
case where $Y$ is a.s.\ empty), which requires the
distribution of $Y$ to be invariant under translation, but stronger than
the \emph{second order intensity-reweighted stationary} (SIRS) of 
\cite{BaddeleyEtAl}. The latter property, in addition to satisfying 
$\bar{\lambda}>0$, requires the random measure 
\[
\Xi=\sum_{(x,t)\in Y}\frac{\delta_{(x,t)}}{\lambda(x,t)}
\] 
to be second order stationary \cite[p.\ 236]{DVJ2}. 

\subsection{Palm measures and conditional intensities}

In order to define a nearest neighbour distance distribution function, we 
need the concept of \emph{reduced Palm measures}. Recall that by assumption
the intensity measure is locally finite. In integral terms, they 
can be defined by the \emph{reduced Campbell-Mecke} formula 
\bea
\label{CMthm}
\E\left[\sum_{(x,t)\in Y}g(x,t,Y\setminus\{(x,t)\})\right]
&=&
\int_{\R^d\times\R}\int_{N}
g(x,t,\varphi)
P^{!(x,t)}(d\varphi)
\lambda(x,t)dx dt
\nn
\\
&=&
\int_{\R^d\times\R}
\E^{!(x,t)}\left[g(x,t,Y)\right]
\lambda(x,t)dx dt
\eea
for any non-negative measurable function $g: \R^d \times R \times N$, 
with the left hand side being infinite if and only if the right hand side 
is infinite, see e.g.\ \cite[Chapter~1.8]{MCbook}. By standard measure 
theoretic arguments \cite{Halmos}, it is possible to find a regular version
such that $P^{!(x,t)}(R)$ is measurable as a function of $(x,t)$ and a probability
measure as a function of $R$. Thus, $P^{!(x,t)}(R)$ may be interpreted as the 
conditional probability of $Y\setminus\{(x,t)\}$ falling in $R\in\mathcal{N}$
given $Y(\{(x,t)\})>0$.

At times we make the further assumption that $Y$ admits a \emph{Papangelou
conditional intensity} $\lambda(\cdot,\cdot;\varphi)$. In effect, we may
then replace expectations under the reduced Palm distribution by 
expectations under $P$. More precisely, (\ref{CMthm}) may be rewritten as 
\begin{align}
\label{Papangelou}
& \E\left[\sum_{(x,t)\in Y}g(x,t,Y\setminus\{(x,t)\})\right]
=
\int_{\R^d\times\R}\E\left[g(x,t,Y) \lambda(x,t;Y)\right] dx dt
\end{align}
for any non-negative measurable function $g\geq0$ on $\R^d\times\R\times N$. 
Equation (\ref{Papangelou}) is referred to as the \emph{Georgii-Nguyen-Zessin} 
formula. We interpret $\lambda(x,t;Y)dx dt$ as the conditional probability of 
finding a space-time point of $Y$ in the infinitesimal region 
$d(x,t)\subseteq\R^d\times\R$, given that the configuration elsewhere coincides
with $Y$. For further details, see e.g.\ \cite[Chapter~1.8]{MCbook}.

\subsection{The generating functional}

For the representation of $J_{\rm inhom}$  in the form (\ref{e:J}), we will need 
the \emph{generating functional} $G(\cdot)$ of $Y$, which is defined as 
\[
G(v) = \E\left[\prod_{(x,t)\in Y}v(x,t)\right]
= \int_{N} \prod_{(x,t)\in\varphi}v(x,t)P(d\varphi)
\]
for all functions $v=1-u$ such that $u:\R^{d}\times\R\rightarrow[0,1]$ is 
measurable with bounded support on $\R^{d}\times\R$. By convention, 
an empty product equals $1$. The generating functional 
uniquely determines the distribution of $Y$ \cite[Theorem 9.4.V.]{DVJ2}. 

Since we assume that the product densities of all orders exist, 
we have that
\begin{align}
\label{SeriesGF}
&G(v) 
=
G(1-u)=
\\
&=
1 + \sum_{n=1}^{\infty} \frac{(-1)^n}{n!} 
\int \cdots \int 
u(x_1,t_1)\cdots u(x_n,t_n)
\rho^{(n)}((x_1,t_1),\ldots,(x_n,t_n)) 
\prod_{i=1}^{n}dx_i dt_i, \nn
\end{align}
provided that the right hand side converges (see \cite[p.\ 126]{SKM}). 
The generating functional of the reduced Palm distribution $P^{!y}$
is denoted by $G^{!y}(v)$.

\section{Spatio-temporal $J$-functions}
\label{SectionJfun}

We now turn to the definition of the inhomogeneous $J$-function 
$J_{\rm inhom}(r,t)$. Before giving the definition in our general context, 
we define a spatio-temporal $J$-function $J(r,t)$ for stationary STPPs. 

\subsection{The stationary $J$-function}

Assume for the moment that $Y$ is stationary. Then we may set,
in complete analogy to the definition in \cite{MCBaddeley},
\bea
\label{StatJfun}
J(r,t) 
= \frac{1-G(r,t)}{1-F(r,t)}
= \frac{\P^{!(0,0)}(Y\cap S_r^t=\emptyset)}{\P(Y\cap S_r^t=\emptyset)}
\eea
for $r,t\geq0$ such that $F(r,t)\neq1$, 
where 
\[
S_r^t=\{(x,s)\in\R^d\times\R:\|x\|\leq r, |s|\leq t\}
\]
and $\P^{!(0,0)}$ is the $P^{!(0,0)}$-reversely induced probability measure on 
$\F$. 

Note that the two equalities in (\ref{StatJfun}) are defining ones and, 
clearly, $G(r,t)$ is the spatio-temporal nearest neighbour distance 
distribution function whereas $F(r,t)$ is the spatio-temporal empty space 
function. 

\subsection{The inhomogeneous $J$-function}
\label{S:defJ}

In this section, we extend the inhomogeneous $J$-function in \cite{MCJfun} 
to the product space $\R^{d}\times\R$ equipped with the supremum metric
$ d(\cdot,\cdot)$.

\begin{definition}
\label{Jfun}
Let $Y$ be an IRMS spatio-temporal point process (cf.~Definition~\ref{d:IRMS}). 
For $r,t\geq0$, let
\beann
J_n(r,t) 
= \int_{S_r^t}\cdots\int_{S_r^t} 
\xi_{n+1}((0,0),(x_1,t_1),\ldots,(x_n,t_n)) 
\prod_{i=1}^{n}dx_i dt_i
\eeann
and set
\begin{align}
\label{JfunSTPP}
J_{\rm inhom}(r,t) = 1 + \sum_{n=1}^{\infty} \frac{(-\bar{\lambda})^n}{n!} J_n(r,t) 
\end{align}
for all spatial ranges $r\geq 0$ and temporal ranges $t\geq 0$ for which 
the series is absolutely convergent. 
\end{definition}

Note that by Cauchy's root test absolute convergence holds for 
those $r,t\geq0$ for which 
$\limsup_{n\rightarrow\infty} \left(
\frac{\bar{\lambda}^n}{n!} |J_n(r,t)|
\right)^{1/n} < 1$.

Let us briefly mention a few special cases. For a Poisson process,
since $\xi_{n+1} \equiv 0$ for $n\geq 1$, $J_{\rm{inhom}}(r, t) \equiv 1$.
Moreover, if $Y$ is stationary, (\ref{JfunSTPP}) reduces to (\ref{StatJfun}).

\subsection{Relationship to $K$-functions}

Spatio-temporal $K$-functions may be obtained as second order approximations 
of $J_{\rm inhom}(r,t)$. To see this, recall that \cite{GabrielDiggle} defines
SIRS (with isotropy) by replacing the second order stationarity of $\Xi$ 
by the stronger condition that the pair correlation function 
$g((x,t),(y,s))=\bar{g}(u,v)$ depends only on the spatial distances 
$u=\|x-y\|$ and the temporal distances $v=|t-s|$. They then introduce 
a spatio-temporal inhomogeneous $K$-function by setting
\begin{align*}
&K_{\rm inhom}(r,t)
=
\int_{S_r^t} \bar{g}(\|x_1\|,|t_1|) d(x_1,t_1)
= 
\omega_d
\int_{-t}^{t} \int_{0}^{r} \bar{g}(u,v)u^{d-1}\,du\,dv,
\end{align*}
where 
$\omega_d/d=\pi^{d/2}/\Gamma(1+d/2)=\kappa_d$ is the volume of the unit ball 
in $\R^d$ (see e.g.\ \cite[p.\ 14]{SKM}). Note that the second equality follows 
from a change to hyperspherical coordinates. If in addition $Y$ is IRMS, 
\begin{align*}
J_{\rm inhom}(r,t) - 1
&=
-
\bar{\lambda}
\bigg(
\omega_d
\int_{-t}^{t}
\int_{0}^{r} \bar{g}(u,v)u^{d-1}\,du\,dv
- \ell(S_r^t)
\bigg)
+
\sum_{n=2}^{\infty} \frac{(-\bar{\lambda})^n}{n!} J_n(r,t)
\\
&\approx - \bar{\lambda}\left(K_{\rm inhom}(r,t) - \ell(S_r^t)\right),
\end{align*}
whereby $K_{\rm inhom}(r,t)$ may be viewed as a (scaled) second order 
approximation of $J_{\rm inhom}(r,t)$. In relation hereto, it should be noted 
that even if the product densities exist only up to some finite order $m$,
we may still obtain an approximation of $J_{\rm inhom}$ by truncating its 
series representation at $n=m$.

Returning now to the original definition of SIRS, where $\bar{\lambda}>0$ 
and $\Xi$ is second order stationary, we may extend the definition of the 
inhomogeneous $K$-function in \cite{BaddeleyEtAl} to the spatio-temporal 
setting by defining 
\begin{align}
\label{InhomK}
K_{\rm inhom}^*(r,t) 
= 
\frac{1}{\ell(A)}
\E\left[
\sum_{(x_1,t_1),(x_2,t_2)\in Y}^{\neq}
\frac
{
\1\{(x_1,t_1)\in A,
\|x_1-x_2\|\leq r,
|t_1-t_2|\leq t\}
}
{\lambda(x_1,t_1)\lambda(x_2,t_2)}
\right]
\end{align}
for $r,t\geq0$ and some set $A=B\times C\in\BB(\R^d\times\R)$ with $\ell(A)>0$. 
By Lemma~\ref{LemmaKfun} below, the definition does not depend on the choice
of $A$.

\begin{lemma}
\label{LemmaKfun}
For any $A=B\times C\in\BB(\R^d\times\R)$ for which $\ell(A)>0$, 
$K_{\rm inhom}^*(r,t)=\mathcal{K}_{\Xi}(S_r^t\setminus\{(0,0)\})$,  
the reduced second factorial moment measure of $\Xi$ evaluated at 
$S_r^t$ (see e.g.\ \cite[Section~12.6]{DVJ2}).
\end{lemma}

\begin{proof}
By the Campbell formula, the intensity measure of $\Xi$ is given by 
\[
\Lambda_{\Xi}(A) = \E \left[
 \sum_{(x,t)\in Y}\frac{1}{\lambda(x,t)} \1_A(x,t) 
\right] = \ell(A),
\]
so it is locally finite and has density $1$. Hence, by 
\cite[Proposition 13.1.IV.]{DVJ2}, 
there exist reduced Palm measures $P_{\Xi}^{!y_1}(R)$, $y_1\in\R^d\times\R$, 
$R\in\mathcal{N}$, such that
\begin{align*}
K_{\rm inhom}^*(r,t) 
&= 
\frac{1}{\ell(A)}
\E\left[
\int_{\R^d\times\R} 
\1_{A}(y_1)
\Xi((y_1 + S_r^t)\setminus\{y_1\}) \,
\Xi(dy_1)
\right]
\\
&
=
\frac{1}{\ell(A)} 
\int_{\R^d\times\R}  \E^{!y} \left[
  1_A(y) \, \Xi( y + S_r^t) 
\right] dy =  \E^{!(0,0)} [\Xi( S_r^t)]
= \mathcal{K}_{\Xi}(S_r^t)
\end{align*}
and this competes the proof.
\end{proof}

It is not hard to see that under the stronger assumptions of
\cite{GabrielDiggle}, $K_{\rm inhom}^*(r,t)=K_{\rm inhom}(r,t)$. 

\section{Representation results}
\label{SectionRepresentation}

Being based on a series of integrals of $n$-point correlation functions, 
Definition \ref{Jfun} highlights the fact that $J_{\rm{inhom}}$
involves interactions of all orders but it is not very convenient in 
practice. The goal of this section is to give representations, which 
are easier to interpret.

\subsection{Representation in terms of generating functionals}

As for purely spatial point processes, we may express 
$J_{\rm{inhom}}$ in terms of the generating functionals $G$ and $G^{!\cdot}$
by making appropriate choices for the functions $v=1-u$ \cite{MCJfun}. Indeed, we may set
\beann
u_{r,t}^{y}(x,s)
=
\frac
{
\bar{\lambda}
\1\{
\|a-x\|\leq r, 
|b-s|\leq t\}
}
{\lambda(x,s)},
\quad y=(a,b)\in\R^d\times\R
,
\eeann
and define the \emph{inhomogeneous spatio-temporal nearest 
neighbour distance distribution function} as 
\bea
\label{PalmGen}
G_{\rm inhom}(r,t) 
&=& 1 - G^{!y}(1-u_{r,t}^{y})
\\
&=& 1 - 
\E^{!(a,b)}
\left[
\prod_{(x,s)\in Y}\left(
1-
\frac{\bar{\lambda}\1\{\|a-x\|\leq r, |b-s|\leq t\}}{\lambda(x,s)}
\right)
\right]
\nn
\eea
and the \emph{inhomogeneous spatio-temporal empty space function} as 
\beann
F_{\rm inhom}(r,t) = 1 -G(1-u_{r,t}^{y})
= 1 - \E\left[
\prod_{(x,s)\in Y}\left(
1-
\frac{\bar{\lambda}\1\{\|a-x\|\leq r, |b-s|\leq t\}}{\lambda(x,s)}
\right)
\right]
\eeann
for $r,t\geq0$, 
under the convention that empty products take the value one. 
Then, the representation theorem below tells us that $G_{\rm inhom}(r,t)$ and $G_{\rm inhom}(r,t)$ 
do not depend on the choice of $y$ and, furthermore, that 
$J_{\rm inhom}(r,t)$ may be expressed through $G_{\rm inhom}(r,t)$ and 
$F_{\rm inhom}(r,t)$. 

\begin{thm}
\label{ReprSTPP}
Let $Y$ be an IRMS spatio-temporal point process and assume that 
\[
\limsup_{n\rightarrow\infty}
\left(
\frac{\bar{\lambda}^n}{n!}
\int_{S_r^t}\cdots\int_{S_r^t} 
\frac{\rho^{(n)}((x_1,t_1),\ldots,(x_n,t_n))}
{\lambda(x_1,t_1)\cdots\lambda(x_n,t_n)}
\prod_{i=1}^{n}dx_i dt_i.
\right)^{1/n}
<1.
\]
Then $G_{\rm inhom}(r,t)$ and $G_{\rm inhom}(r,t)$ are $\ell$-almost everywhere constant with 
respect to $y=(a,b)\in\R^d\times\R$ and 
$J_{\rm inhom}(r,t)$ in expression (\ref{JfunSTPP}) can be written as
\beann
J_{\rm inhom}(r,t)
=
\frac{1-G_{\rm inhom}(r,t)}{1-F_{\rm inhom}(r,t)}
\eeann
for all $r,t\geq0$ such that $F_{\rm inhom}(r,t)\neq1$. 
\end{thm}

\begin{proof}
From expression (\ref{SeriesGF}) it follows that 
\begin{align*}
&G(1-u_{r,t}^{y})=
1 + \sum_{n=1}^{\infty} \frac{(-\bar{\lambda})^n}{n!} 
\int_{y+S_r^t} \cdots \int_{y+S_r^t}
\frac{\rho^{(n)}((x_1,t_1),\ldots,(x_n,t_n))}
{\lambda(x_1,t_1)\cdots\lambda(x_n,t_n)}
\prod_{i=1}^{n}dx_i dt_i, 
\end{align*}
since, by assumption, the series on the right hand side is absolutely
convergent. 
Note that $G(1-u_{r,t}^{y})$ is a constant for almost
all $(a,b)\in\R^d\times\R$ by the IRMS assumption. 
Furthermore, by an inclusion-exclusion argument, 
\begin{align*}
G^{!y}(1-u_{r,t}^{y})
=
1 + \sum_{n=1}^{\infty} \frac{(-\bar{\lambda})^n}{n!} 
\E^{!(a,b)}
\left[
\sum_{(x_1,t_1),\ldots,(x_n,t_n)\in Y}^{\neq}
\prod_{i=1}^{n}
\frac{\1\{\|a-x_i\|\leq r, |b-t_i|\leq t\}}{\lambda(x_i,t_i)}
\right],
\end{align*}
which is well defined by the local finiteness of $Y$ (the factor $1/n!$ 
removes the implicit ordering of $\sum^{\neq}$). 
To show the independence of the choice of $(a,b)$, 
for any bounded $A = B\times C\in\BB(\R^{d}\times\R)$ and any $n\geq1$, 
consider now the non-negative measurable function 
\beann
g_{r,t}^A((a,b),\varphi) 
= 
\frac{\1\{(a,b)\in A\}}{\lambda(a,b)}
\sum_{(x_1,t_1),\ldots,(x_n,t_n)\in\varphi}^{\neq}
\prod_{i=1}^{n}
\frac{
\1\{\|a-x_i\|\leq r, |b-t_i|\leq t\}
}{\lambda(x_i,t_i)}.
\eeann
By rewriting the expression for $g_{r,t}^A((a,b),Y\setminus\{(a,b)\})$, 
recalling the Campbell formula (\ref{ExpProdDens}) and taking 
the translation invariance of $\xi_n$ into consideration, we obtain
\beann
&&\E
\left[
\sum_{(a,b)\in Y}
g_{r,t}^A((a,b),Y\setminus\{(a,b)\})
\right]
=
\\
&=&
\E
\left[
\sum_{(a,b),(x_1,t_1),\ldots,(x_n,t_n)\in Y}^{\neq}
\frac{\1\{(a,b)\in A\}}{\lambda(a,b)}
\prod_{i=1}^{n}
\frac{\1\{\|a-x_i\|\leq r, |b-t_i|\leq t\}}{\lambda(x_i,t_i)}
\right]
\\
&=&
\int_{B\times C} 
\Bigg(
\int_{S_r^t+y} \cdots \int_{S_r^t+y}
\frac{\rho^{(n+1)}((a,b),y_1, \ldots, y_n)}{
\lambda(a,b)\lambda(y_1)\cdots\lambda(y_n)} 
dy_1 \cdots dy_n
\Bigg)da db
\\
&=&
\int_{A} 
\left(
\int_{S_r^t} \cdots \int_{S_r^t}
\frac{\rho^{(n+1)}((0,0),(x_1,t_1),\ldots,(x_n,t_n))}{\lambda(0,0)\lambda(x_1,t_1)\cdots\lambda(x_n,t_n)} 
\prod_{i=1}^{n}dx_i dt_i
\right) da db.
\eeann
On the other hand, by the reduced Campbell--Mecke formula (\ref{CMthm}), 
\begin{align*}
&\E
\left[
\sum_{(a,b)\in Y}
g_{r,t}^A((a,b),Y\setminus\{(a,b)\})
\right]
=
\\
&=
\int_{A}
\E^{!(a,b)}
\left[
\sum_{(x_1,t_1),\ldots,(x_n,t_n)\in Y}^{\neq}
\prod_{i=1}^{n}
\frac{
\1\{\|a-x_i\|\leq r, |b-t_i|\leq t\}
}{\lambda(x_i,t_i)}
\right]
dx ds.
\end{align*}
Hereby the two expressions above equal each other for all $A$ and consequently the integrands are equal for $\ell$-almost all $y=(a,b)\in\R^d\times\R$.
Hence, for almost all $y=(a,b)\in\R^d\times\R$,  
\begin{align*}
&G^{!y}(1-u_{r,t}^{y})= 1+
\\
&+ \sum_{n=1}^{\infty} \frac{(-\bar{\lambda})^n}{n!} 
\int_{S_r^t}\cdots\int_{S_r^t}
\frac{\rho^{(n+1)}((0,0),(x_2,t_2),\ldots,(x_{n+1},t_{n+1}))}{
\lambda(0,0)\lambda(x_2,t_2)\cdots\lambda(x_{n+1},t_{n+1})} 
dx_2 dt_2\cdots dx_{n+1} dt_{n+1}
\\
&=
1 + \sum_{n=1}^{\infty} \frac{(-\bar{\lambda})^n}{n!} 
\int_{S_r^t}\cdots\int_{S_r^t}
\sum_{k=1}^{n+1}
\sum_{D_1,\ldots,D_k}
\prod_{j=1}^{k}
\xi_{|D_j|}(\{z_i:i\in D_j\})
dz_2\cdots dz_{n+1},
\end{align*} 
where $z_{1} \equiv (0,0)$ and $z_{i}=(x_i,t_i)$, $i=2,\ldots,n+1$. 
Recall that $\sum_{D_1,\ldots,D_k}$ is a sum over all possible $k$-sized
partitions $\{D_1,\ldots,D_k\}$, $\emptyset\neq D_j\in\mathcal{P}_{n+1}$, 
where $\mathcal{P}_{n+1}$ denotes the power set of $\{1,\ldots,n+1\}$.

With the convention that $\sum_{k=1}^{0}=1$, 
we may split the above expression into terms based on whether the index
 sets $D_j$ contain the index $1$ (i.e.\ whether $\xi_{|D_j|}$ includes 
$z_{1} \equiv (0,0)$) to obtain 
\beann
G^{!y}(1-u_{r,t}^{y}) 
&=& 1+ 
\sum_{n=1}^{\infty} \frac{(-\bar{\lambda})^n}{n!} 
\sum_{D\in\mathcal{P}_n}
\overbrace{
\int_{S_r^t}\cdots\int_{S_r^t}
\xi_{|D|+1}(0,z_1,\ldots,z_{|D|})
dz_1\cdots dz_{|D|}
}^{=J_{|D|}(r,t)}
\\
&&\times
\sum_{k=1}^{n-|D|}
\sum_{\substack{D_1,\ldots,D_k\neq\emptyset\text{ disjoint}\\ \cup_{j=1}^{k}D_j = \{1,\ldots,n\}\setminus D}}
\prod_{j=1}^{k} 
I_{|D_j|},
\eeann
where 
$
I_{n} =
\int_{S_r^t}\cdots\int_{S_r^t}
\xi_{n}(z_2,\ldots,z_{n+1})
dz_2\cdots dz_{n+1}.
$ 
The right hand side of the above expression may be written as 
\[
\left(
1+ \sum_{n=1}^{\infty} \frac{(-\bar{\lambda})^n}{n!} J_{n}(r,t)
\right)
\Bigg(
1+ \sum_{m=1}^{\infty} \frac{(-\bar{\lambda})^m}{m!} 
\sum_{k=1}^{m}
\sum_{\substack{D_1,\ldots,D_k\neq\emptyset\text{ disjoint}\\ \cup_{j=1}^{k}D_j = \{1,\ldots,m\}}}
\prod_{j=1}^{k} 
I_{|D_j|}
\Bigg),
\]
which equals
\[
J_{\rm inhom}(r,t)
\Bigg(
1+ \sum_{m=1}^{\infty} \frac{(-\bar{\lambda})^m}{m!} 
\int_{S_r^t} \cdots \int_{S_r^t}
\frac{\rho^{(m)}((x_1,t_1),\ldots,(x_m,t_m))}
{\lambda(x_1,t_1)\cdots\lambda(x_m,t_m)}
\prod_{i=1}^{m}dx_i dt_i, 
\Bigg)
\]
by Fubini's theorem and the definition of the $n$-point correlation functions. 
The absolute convergence of the individual sums in the above product imply 
the absolute convergence of 
$G^{!y}(1-u_{r,t}^{y}) = J_{\rm inhom}(r,t) G(1-u_{r,t}^{0})$ and this in turn 
completes the proof. 
\end{proof}

The intuition behind $G_{\rm inhom}(r,t)$ and $F_{\rm inhom}(r,t)$ is best
seen when $Y$ is stationary. In this case
$
u_{r,t}^{0}(x,s)
=
\1\{ (x,s) \in S_r^t \}
$
and hence
\[
F_{\rm inhom}(r,t)
= 1-\E\left[\prod_{(x,s)\in Y}\1\{(x,s)\notin S_r^t\}\right]
= 1-\P(Y\cap S_r^t=\emptyset)
=F(r,t),
\]
the empty space function in expression (\ref{StatJfun}). 
Similarly, $G_{\rm inhom}(r,t)$ reduces to the distribution function
of the nearest neighbour distance when $Y$ is stationary, and the
$J$-function is indeed a generalisation of (\ref{e:J}).

\subsection{Representation in terms of conditional intensities}

Some families of point processes, notably Gibbsian ones \cite{MCbook}, are 
defined in terms of their Papangelou conditional intensity 
$\lambda(\cdot,\cdot;\cdot)$.
Below we show that for such processes, $J_{\rm inhom}$ may be represented 
in terms of $\lambda(\cdot,\cdot;\cdot)$.

\begin{thm}
\label{ThmPapangelou}
Let the assumptions of Theorem~\ref{ReprSTPP} hold and assume, in addition, 
that $Y$ admits a conditional intensity $\lambda(\cdot,\cdot;\cdot)$. 
Write 
$W_{(a,b)}(Y)
=
\prod_{(x,s)\in Y}\left(
1 - u_{r,t}^{(a,b)}(x,s)
\right)
$.
Then 
$\E\left[\lambda(a,b;Y) W_{(a,b)}(Y) / \lambda(a,b) \right]  > 0$ implies
$E[W_{(a,b)}(Y)] > 0$ and
\beann
J_{\rm inhom}(r,t)
=
\E\left[\frac{\lambda(a,b;Y)}{\lambda(a,b)}W_{(a,b)}(Y)\right]
/\E[W_{(a,b)}(Y)]
\eeann
for almost all $(a,b)\in\R^d\times\R$. 
\end{thm}

\begin{proof}
We already know that $\E[ W_{(a,b)}(Y) ]$ is a constant for almost
all $(a,b)\in\R^d\times\R$.
Since $0\leq W_{(a,b)}(Y)\leq1$, if $E[W_{(a,b)}(Y)]=0$, then 
$W_{(a,b)}(Y)\stackrel{a.s.}{=}0$ and this implies that 
$\E[\frac{\lambda(a,b;Y)}{\lambda(a,b)}W_{(a,b)}(Y)]  =0$. 
Note that by the Georgii--Nguyen--Zessin formula (\ref{Papangelou}) in combination with 
(\ref{CMthm}), for any bounded 
$A = B\times C\in\BB(\R^{d}\times\R)$,  
\begin{align*}
&\int_{A}
\E^{!(a,b)}
\left[
\frac{1}{\lambda(a,b)}
\prod_{(x,s)\in Y}\left(
1-
\frac{\bar{\lambda}\1\{\|a-x\|\leq r, |b-s|\leq t\}}{\lambda(x,s)}
\right)
\right]
\lambda(a,b)
da db
\\
&=
\int_{A}
\E
\left[
\frac{\lambda(a,b;Y)}{\lambda(a,b)}
\prod_{(x,s)\in Y}\left(
1-
\frac{\bar{\lambda}\1\{\|a-x\|\leq r, |b-s|\leq t\}}{\lambda(x,s)}
\right)
\right]
da db,
\end{align*}
whereby the integrands are equal for almost all $(a,b)\in\R^d\times\R$ 
and the claim follows from Theorem~\ref{ReprSTPP}.
\end{proof}

Since $\E[\lambda(a,b;Y)] = \lambda(a,b)$, we immediately see that
$$
J_{\rm inhom}(r,t)\geq1 
\Longleftrightarrow
\Cov\left(
\lambda(a,b;Y)
, W_{(a,b)}(Y)\right)\geq0
$$ 
and 
$$
J_{\rm inhom}(r,t)\leq1 
\Longleftrightarrow
\Cov\left(
\lambda(a,b;Y)
, W_{(a,b)}(Y)\right)\leq0.
$$ 
In words, for clustered point processes, $\lambda(a,b; Y)$ tends to
be large if $(a,b)$ is near to points of $Y$ whereas $W_{(a,b)}(Y)$ tends 
to be large when there are few points of $Y$ close to $(a,b)$. Thus, in 
this case, the two random variables are negatively correlated and the 
$J$-function is smaller than one. A dual reasoning applies for regular 
point processes, but \cite{Bedford} warns against drawing too strong
conclusions.

\subsection{Scaling}

In expression (\ref{JfunSTPP}), we consider distances on the spaces 
$\R^d$ and $\R$ separately. Instead, we could have used the supremum
distance on $\R^d\times\R$ and the closed $d$-metric balls 
$B[0,r] = S_r^r$, $r\geq0$ to define
$J_n(r)=J_n(r,r)$ and
\begin{equation}
\label{e:J1}
J_{\rm inhom}(r) = 1 + \sum_{n=1}^{\infty} \frac{(-\bar{\lambda})^n}{n!} J_n(r).
\end{equation}
When the pair correlation function only depends on the spatial and temporal 
distances, set
$K_{\rm inhom}^*(r)=K_{\rm inhom}^*(r,r)$ and $K_{\rm inhom}(r)=K_{\rm inhom}(r,r)$,
whence $K_{\rm inhom}^*(r)=K_{\rm inhom}(r)$ and 
$
J_{\rm inhom}(r) -1
\approx
- \bar{\lambda}(K_{\rm inhom}(r) - \ell(B[0,r])).
$

In the remainder of this subsection, we argue that (\ref{JfunSTPP}) may
be obtained from (\ref{e:J1}) by scaling. 
Let $c=(c_S, c_T)\in(0,\infty)^2$ and apply the bijective transformation 
$(y,s)\mapsto (c_S y, c_T s)$ to each point of the IRMS spatio-temporal
point process  $Y$ to obtain 
\[
cY=\sum_{(y,s)\in Y}\delta_{(c_S y, c_T s)}.
\]
Through a change of variables and the Campbell formula, one obtains 
\[
\rho_{cY}^{(n)}((x_1,t_1),\ldots,(x_n,t_n))
= c_S^{-dn}c_T^{-n} \rho^{(n)}((x_1/c_S,t_1/c_T),\ldots,(x_n/c_S,t_n/c_T)),
\]
so that 
$\lambda_{cY}(x,t)=c_S^{-d}c_T^{-1}\lambda(x/c_S,t/c_T)$ and 
$\bar{\lambda}_{cY}=\inf_{(x,t)}\lambda_{cY}(x,t) = c_S^{-d}c_T^{-1}\bar{\lambda}$.
Hence,
\[
\xi_n^{cY}((x_1,t_1),\ldots,(x_n,t_n)) = 
\xi_n((x_1/c_S,t_1/c_T),\ldots,(x_n/c_S,t_n/c_T))
\]
whence $cY$ is IRMS if and only if $Y$ is and, whenever well-defined,
\begin{equation}
\label{Jthin}
J_{\rm inhom}^{cY}(r,t) 
=
J_{\rm inhom}\left(\frac{r}{c_S}, \frac{t}{c_T} \right).
\end{equation}
In conclusion, by taking $c_S=1$, and $c_T=r/t$, any $J_{\rm inhom}(r,t)$
may be obtained from $J_{\rm inhom}(r)$ through scaling.

\section{Examples of spatio-temporal point processes}
\label{SectionExamples}

Below we will consider three families of models, each representing a 
different type of interaction.

\subsection{Poisson processes}
\label{SectionPoisson}

The inhomogeneous Poisson processes may be considered the benchmark 
scenario for lack of interaction between points. As we saw in
Section~\ref{S:defJ}, for a Poisson process  $J_{\rm inhom}(r,t)\equiv1$.
Alternative proofs may be obtained from the representation Theorems
\ref{ReprSTPP} and \ref{ThmPapangelou}, by noting that the Palm distributions
equal $P$ by Slivnyak's theorem \cite{SchneiderWeil}, or that the intensity 
function and the Papangelou conditional intensity coincide almost everywhere.

\subsection{Location dependent thinning}
\label{SectionThinning}

Given a stationary STPP $Y$ with product densities $\rho^{(n)}$, $n\geq1$, 
intensity $\lambda>0$ and $J$-function $J(r,t)$, consider some measurable function 
$p:\R^d\times\R\rightarrow(0,1]$ with $\bar{p}=\inf_{(x,t)}p(x,t)>0$. 
Location dependent thinning of $Y$ is the scenario in which a point
$(x,t) \in Y$ is retained with probability $p(x,t)$. Denote the
resulting thinned process by $Y_{\rm th}$.

The product densities of $Y_{\rm th}$ are 
\[
\rho_{\rm th}^{(n)}((x_1,t_1),\ldots,(x_n,t_n)) = 
\rho^{(n)}((x_1,t_1),\ldots,(x_n,t_n))\prod_{i=1}^{n}p(x_i,t_i)
\]
by \cite[Section~11.3]{DVJ2}, whereby 
$\lambda_{\rm th}(x,t)=\lambda p(x,t)>0$ and 
the $n$-point correlation functions of $Y_{\rm th}$ and $Y$ coincide. 
Hence, $Y_{\rm th}$ is IRMS with 
$\bar{\lambda}=\inf_{(x,t)}\lambda_{\rm th}(x,t)=\lambda\bar{p}$ and 
\beann
J_{\rm inhom}^{\rm th}(r,t) = 1 + 
\sum_{n=1}^{\infty} \frac{(-\lambda\bar{p})^n}{n!} J_n(r,t)
\eeann
for all $r,t\geq0$ for which the series converges. 
Here $J_n(r,t)$ is the $n$-th coefficient in the series expansion (\ref{JfunSTPP})) of the $J$-function of the original process $Y$. 

A more informative expression for $J_{\rm inhom}^{\rm th}$ can be obtained by noting that,
by \cite[Eq.~(5.3)--(5.4)]{SKM}, the generating functional of $Y_{\rm th}$ is 
given by $G_{\rm th}(v)=G(1-p+pv)$, where $G(\cdot)$ is the generating 
functional of $Y$. Hence, since applying thinning to the reduced Palm 
distribution of $Y$ is equivalent to Palm conditioning in the thinned 
process, 
\begin{align*}
J_{\rm inhom}^{\rm th}(r,t) 
=
\frac{G_{\rm th}^{!(0,0)}(1-\bar{p}\1\{\cdot\in S_r^t\}/p)}{G_{\rm th}(1-\bar{p}\1\{\cdot\in S_r^t\}/p)}
=
\frac{G^{!(0,0)}(1-\bar{p}\1\{\cdot\in S_r^t\})}{G(1-\bar{p}\1\{\cdot\in S_r^t\})}
=
\frac{\E^{!(0,0)}[(1-\bar{p})^{Y(S_r^t)}]}{\E[(1-\bar{p})^{Y(S_r^t)}]}
\end{align*}
when Theorem~\ref{ReprSTPP} applies. 

When a Papangelou conditional intensity exists for $Y$, by recalling that $\lambda=\lambda(x,t)=\E[\lambda(x,t;Y)]$ and applying the combination of (\ref{CMthm}) and (\ref{Papangelou}) to the restriction of the function $g(a,b,Y)= (1-\bar{p})^{Y((a,b)+S_r^t)}$ to arbitrary bounded space-time domains, the previous expression becomes
\begin{align}
\label{JfunThin}
J_{\rm inhom}^{\rm th}(r,t) 
=
\frac{\E[\lambda(0,0;Y) (1-\bar{p})^{Y(S_r^t)}]}
{\lambda\E[(1-\bar{p})^{Y(S_r^t)}]}
.
\end{align}

\subsubsection{Thinned hard core process}

The spatio-temporal hard core process is a stationary STPP defined through its Papangelou 
conditional intensity 
\bea
\label{PapangelouStrauss}
\lambda_Y(a,b;Y) = \beta \, \1\{Y\cap ((a,b)+S_{R_S}^{R_T})=\emptyset\}
=\beta \prod_{(x,t)\in Y}\1\{(x,t)-(a,b)\notin S_{R_S}^{R_T}\},
\eea
where $(a,b)\in\R^d\times\R$. Moreover, $\beta>0$ is a model parameter and $R_S>0$ 
and $R_T>0$ are, respectively, the spatial hard core distance and the temporal hard core distance. In words,
since realisations a.s.\ do not contain points that violate the 
spatial and temporal hard core constraints, i.e.\ 
$\P^{!(0,0)}(Y(S_{R_S}^{R_T})>0)=0$, there is inhibition. 

By thinning $Y$ with some suitable measurable retention function 
$p:\R^d\times\R\rightarrow(0,1]$, $\bar{p}=\inf_{(x,t)}p(x,t)>0$, 
we obtain an IRMS hard core STPP $Y_{\rm th}$.

\begin{lemma}
For a hard core process $Y$, $\beta/\lambda\geq1$. 
If either $(r,t)\in[0,R_S]\times[0,R_T]$ or $(r,t)\in[R_S,\infty)\times[R_T,\infty)$, $J(r,t)$ is increasing in $r$ and $t$. 
Moreover, when $(r,t)\in[0,R_S]\times[0,R_T]$ we have that $1\leq J(r,t)\leq\beta/\lambda$ and when $(r,t)\in[R_S,\infty)\times[R_T,\infty)$, $J(r,t)=\beta/\lambda$. 
When $R_T=R_S=R>0$, so that $S_{R_S}^{R_T}=B[0,R]$, $J(r)=J(r,r)$ is increasing and satisfies $1\leq J(r)<\beta/\lambda$ for $r\in[0,R)$ and $J(r)=\beta/\lambda$ for $r\geq R$. 

For a thinned hard core process, $J_{\rm inhom}^{\rm th}(r,t)\geq1$ for $r\leq R_S$ and $t\leq R_T$. 
\end{lemma}

\begin{proof}
Noting that $\lambda=\lambda(0,0)=\E[\lambda_Y(0,0;Y)]=\beta\P(Y\cap S_{R_S}^{R^T}=\emptyset)\leq\beta$ we find that $\beta/\lambda\geq1$. 
Furthermore, 
through Theorem~\ref{ThmPapangelou} and expression (\ref{PapangelouStrauss}) we obtain
\begin{align*}
&J(r,t)
=
\frac{\E\left[\lambda_Y(0,0;Y) \1\{Y\cap S_r^t=\emptyset\} \right]}
{\lambda(0,0) \E\left[\1\{Y\cap S_r^t=\emptyset\}\right]}
=
\frac{\beta}{\lambda}
\frac{\P\left(Y\cap S_{R_S}^{R_T}=\emptyset, Y\cap S_r^t=\emptyset \right)}
{\P\left(Y\cap S_r^t=\emptyset\right)}
.
\end{align*}
Hence, when both $r\geq R_S$ and $t\geq R_T$ we have that $S_{R_S}^{R_T} \subseteq S_r^t$ and consequently $J(r,t)=\beta/\lambda$. 
Moreover, when $r\leq R_S$ and $t\leq R_T$, so that $S_r^t\subseteq S_{R_S}^{R_T}$, expression (\ref{StatJfun}) gives us $J(r,t)=1/(1-F(r,t))$, which is increasing in both $r\in[0,R_S]$ and $t\in[0,R_T]$ and satisfies $J(r,t)\geq1$. 
By setting $r=R_S$ and $t=R_T$ we confirm that $J(r,t)=\beta/\lambda\geq1$. 

Specialising to $J(r)=J(r,r)$, when $r\leq R$ we have that $J(r)=1/(1-F(r,r))$, which is increasing to $\beta/\lambda$, and when $r>R$,  $J(r)=\beta/\lambda$. 

When $Y$ is thinned and $r\leq R_S$ and $t\leq R_T$, 
\begin{align*}
J_{\rm inhom}^{\rm th}(r,t) = \frac{\E^{!(0,0)}[(1-\bar{p})^{Y(S_r^t)}]}{\E[(1-\bar{p})^{Y(S_r^t)}]}
\geq
\frac{\E^{!(0,0)}[(1-\bar{p})^{Y(S_{R_S}^{R_T})}]}{\E[(1-\bar{p})^{Y(S_r^t)}]}
=
\frac{1}{\E[(1-\bar{p})^{Y(S_r^t)}]}
\geq
1.
\end{align*}

\end{proof}

\subsection{Log-Gaussian Cox processes}
\label{SectionCox}

Our final example concerns spatio-temporal versions of log-Gaussian Cox 
processes (see e.g.\ \cite{ColeJone91,MollerSyversveen,Rath96}). 
In words, these models are spatio-temporal Poisson processes for which 
the intensity functions are given by realisations of log-Gaussian random 
fields \cite{Adler,AdlerTaylor}.

Recall that a Gaussian random field is completely determined by its mean function
$\mu(x,t)$ and its covariance function 
$C((x,t),(y,s))$, $(x,t),(y,s)\in\R^d\times\R$, and that by Bochner's theorem $C$ must be positive definite (see e.g.\ \cite[Section~2.4]{Handbook}). Now, a spatio-temporal 
log-Gaussian Cox process $Y$ has random intensity function given by
$$
\exp\left\{ \mu(x,t) + Z(x,t)\right\}, \quad (x,t)\in\R^d\times\R, 
$$
where $Z=\{Z(x,t)\}_{(x,t)\in\R^d\times\R}$, is a zero-mean spatio-temporal 
Gaussian random field. Note that the variance function of $X$ is given 
by $\sigma^2(x,t)=C((x,t),(x,t))$ and the correlation function by 
$r((x,t),(y,s)) = C((x,t),(y,s))/(\sigma(x,t)\sigma(y,s))$. 
By \cite[Section~6.2]{DVJ1} or \cite[Section~5.2]{SKM},
\[
\frac{\rho^{(n)}((x_1,t_1),\ldots,(x_n,t_n))}{
\lambda(x_1,t_1)\cdots\lambda(x_n,t_n)}
=\exp\left\{\sum_{i<j} C((x_i,t_i),(x_j,t_j))\right\}
\]
and the intensity function of $Y$ is
\bea
\label{IntesityLGCP}
\lambda(x,t)=\exp\left\{ \mu(x,t)+\sigma^2(x,t)/2\right\}.
\eea
Therefore, if $\inf_{(x,t)}\exp\{\mu(x,t)\} > 0$ so that $\lambda(x,t)$ is bounded
away from zero, under the additional condition that 
$C((x,t),(y,s))= C(x-y,t-s)$, $Y$ is IRMS.
In this case, $\sigma^2(x,t) = C(0,0)= \sigma^2$ and $Z$ is stationary.
To exclude trivial cases, we shall assume that $\sigma^2 > 0$.

Before we proceed, note that we must impose conditions on $r$ 
to ensure that the function $\exp\{ \mu(x,t) + Z(x,t) \}$ is integrable 
and defines a locally finite random measure. Further details are
given in the Appendix. Henceforth, we will assume that $\mu(x,t)$ is 
continuous and bounded with $\bar{\mu}=\inf_{(x,t)}\mu(x,t)>-\infty$, 
so that $\bar{\lambda}=\exp\{\bar{\mu}+\sigma^2/2\}$,  
and that $r$ is such that $Z$ a.s.\ has continuous sample paths.
Combining \cite[Proposition~6.2.II]{DVJ1} with \cite[(5.35)]{SKM}, we obtain,
under the assumptions of Theorem \ref{ReprSTPP}, 
\beann
J_{\rm inhom}(r,t) 
&=&
\frac
{\E\left[\e^{Z(0,0)}
\exp\left\{- \int_{S_r^t}\e^{\bar{\mu} + Z(x,s)}dx ds\right\}\right]}
{\E[\e^{Z(0,0)}]
\E\left[\exp\left\{- \int_{S_r^t}\e^{\bar{\mu} + Z(x,s)}dx ds\right\}\right]} 
\eeann
upon noting that the Palm distribution of the driving random measure of our
log-Gaussian Cox process $Y$ is $\exp\{Z(x,t)\}$-weighted. 
Note here that the Papangelou conditional intensity of $Y$ exists and is given by 
$\lambda(x,t;Y)=\E[\exp\{\mu(x,t) + Z(x,t)\} | Y]$ 
(see e.g.\ \cite{MollerWaagepetersenArticle}).

\begin{lemma}
For a log-Gaussian Cox processes, 
when the above conditions are imposed on $\mu$ and $C$, $J_{\rm inhom}(r,t)\leq1$ for all $r,t\geq0$. 
\end{lemma}

\begin{proof}
First, observe that $J_{\rm inhom}(r,t)\leq1$ is equivalent to 
$\Cov(\e^{Z(0,0)}, 
\e^{- \e^{\bar{\mu}} \int_{S_r^t}\e^{Z(x,s)}dx ds}
)\leq0$. 
Further, note that by the a.s.\ sample path continuity of $Z$, 
$$
\e^{- \e^{\bar{\mu}} \int_{S_r^t}\e^{Z(x,s)}dx ds}
\stackrel{a.s.}{=}
\lim_{n\rightarrow\infty}  \e^{- \e^{\bar{\mu}} \sum_{(x_i,s_i)\in S(n)} c_{i,n}\e^{Z(x_i,s_i)}},
$$ 
where $S(n)\subseteq S_r^t$, $n\geq1$, are Riemann partitions. 
Since $Z$ has positive correlation function, Pitt's theorem \cite{Pitt} tells us that 
$Z$ is an associated family of random variables. Hereby 
$\Cov(\e^{Z(0,0)}, 
\exp\{- \e^{\bar{\mu}} \sum_{(x_i,s_i)\in S(n)} c_{i,n} \e^{Z(x_i,s_i)}\}
)\leq0$ for any $n\geq1$ 
and 
the result follows 
from taking the limit in the last covariance and applying dominated convergence. 
\end{proof}

\section{Estimation}
\label{SectionEstimation}

Assume that we observe an IRMS STPP $Y$ within some compact spatio-temporal 
region $W_S\times W_T\subseteq\R^d\times\R$ and obtain the realisation 
$\{(x_i,t_i)\}_{i=1}^{m}$, $m=Y(W_S\times W_T)$. 
The goal of this section is to derive estimators for 
$G_{\rm inhom}(r,t)$, $F_{\rm inhom}(r,t)$ and $J_{\rm inhom}(r,t)$.
In order to deal with possible edge effects we will apply a minus 
sampling scheme \cite{SKM,CronieSarkka}.  For clarity of
exposition, we assume that the intensity function is known. 

Denote the boundaries of $W_S$ and $W_T$ by $\partial W_S$ and $\partial W_T$, 
respectively. Further, write 
$
W_S^{\ominus r}$ $ = \{x\in W_S:d_{\R^d}(x,\partial W_S)\geq r\}$
$= \{x\in W_S: x + B_{\R^d}[0,r]\subseteq W_S\}$
for the eroded spatial domain and, similarly, let 
$W_T^{\ominus t}$ $ = \{s\in W_T:d_{\R}(s,\partial W_T)\geq t\}$.
For given $r,t\geq0$, we define an estimator of $1- G_{\rm inhom}(r,t)$ by
\begin{equation}
\label{NumeratorEst}
\frac
{1}{|Y\cap(W_S^{\ominus r}\times W_T^{\ominus t})|}
\sum_{(x',s')\in Y\cap(W_S^{\ominus r}\times W_T^{\ominus t})}
\left[
\prod_{(x,s)\in (Y\setminus\{(x',s')\})\cap((x',s')+S_r^t)}
\left(1 - \frac{\bar{\lambda}}{\lambda(x,s)}\right)
\right]
\end{equation}
and, given a finite point grid $L\subseteq W_S\times W_T$, we 
estimate $1 - F_{\rm inhom}(r,t)$ by
\begin{equation}
\label{DenominatorEst}
\frac
{1}{|L\cap(W_S^{\ominus r}\times W_T^{\ominus t})|}
\sum_{l\in L\cap(W_S^{\ominus r}\times W_T^{\ominus t})}
\left[
\prod_{(x,s)\in Y\cap(l+S_r^t)}
\left(1 - \frac{\bar{\lambda}}{\lambda(x,s)}\right)
\right].
\end{equation}
The ratio of (\ref{NumeratorEst}) and (\ref{DenominatorEst}) is
an estimator of $J_{\rm inhom}(r,t)$, cf.\ Theorem~\ref{ReprSTPP}.

\begin{thm}
Under the conditions of Theorem~\ref{ReprSTPP}, the estimator 
(\ref{DenominatorEst}) is unbiased and (\ref{NumeratorEst}) is ratio-unbiased. 
\end{thm}

\begin{proof}
We start with (\ref{NumeratorEst}) and note that 
$
\E[Y(W_S^{\ominus r}\times W_T^{\ominus t})] = $
$\Lambda(W_S^{\ominus r}\times W_T^{\ominus t})$.
By the reduced Campbell-Mecke formula (\ref{CMthm}),
\begin{align*}
&
\E\left[\sum_{(x',s')\in Y\cap(W_S^{\ominus r}\times W_T^{\ominus t})}
\prod_{(x,s)\in (Y\setminus\{(x',s')\})\cap((x',s')+S_r^t)}
\left(1 - \frac{\bar{\lambda}}{\lambda(x,s)}\right)
\right]
=
\\
&=
\int_{W_S^{\ominus r}\times W_T^{\ominus t}}
\E^{!(x',s')}
\left[
\prod_{(x,s)\in Y
}
\left(1 - \frac{\bar{\lambda}}{\lambda(x,s)}
\1\{(x- x', s- s')\in S_r^t\}
\right)
\right]
\lambda(x',s')
dx' ds'.
\end{align*}
By (\ref{PalmGen}), the expectation is equal to $G^{!0}(1-u_{r,t}^{0})$, 
from which the claimed ratio-unbiasedness follows. 

Turning to (\ref{DenominatorEst}), unbiasedness follows from the
assumed translation invariance of the $\xi_n$s and equation (\ref{SeriesGF}) 
under the conditions of Theorem~\ref{ReprSTPP}.
\end{proof}

In practice, the intensity function $\lambda(x,s)$ is not known. 
Therefore an estimator  $\widehat{\lambda}(x,s)$ will have to 
be obtained and then used as a plug-in in the above estimators. 
E.g.\ \cite{GabrielDiggle} considers kernel estimators for $\lambda(x,s)$ 
but stresses, however, that care has to be taken when 
$\widehat{\lambda}(x,s)$ is close to 0, since a change of bandwidth may 
cause $\widehat{\lambda}(x,s)=0$ for some $(x,s)$, which would be in 
violation of the assumption that $\bar{\lambda}>0$. 

\section{Numerical evaluations}
\label{SectionData}

In this section, we use the inhomogeneous $J$-function to quantify the 
interactions in a realisation of each of the three models discussed in 
Section~\ref{SectionExamples}. In order to do so, we work mostly in 
{\tt R} and exploit functions in the package \verb|spatstat| 
\cite{BaddeleyTurner}, in which versions of all summary statistics
discussed in this paper have already been implemented for purely spatial 
point processes, both in the general and the stationary case; the 
spatio-temporal $K$-function has been implemented in \verb|stpp|
\cite{GabrielDiggle}. To simulate log Gaussian Cox processes we use the 
package \verb|RandomFields| \cite{Schlather}. Realisations of spatio-temporal 
hard core processes can be obtained using the {\tt C++} library 
{\tt MPPLIB} \cite{Steenbeek}.

Throughout this section, realisations will be restricted to the observation
window $W_S\times W_T = [0,1]^2\times[0,1]$. The intensity function is either
known, or known up to a constant (for the thinned hard core process). Hence, 
since (\ref{NumeratorEst})--(\ref{DenominatorEst}) are defined in terms of
the ratio $\bar \lambda / \lambda(x,s)$, there is no need to plug in intensity 
function estimators.

\subsection{Poisson processes}

Consider a Poisson process $Y$ on $\R^2 \times \R$ with intensity function
$\lambda(x,y,t) = 750\e^{-1.5(y+t)}$ as in Section~\ref{SectionPoisson}. 
Note that $\bar{\lambda} = 750\e^{-3} \approx 37.34$ and that the expected number 
of observed points of $Y$ in $W_S\times W_T$ is $750(1-\e^{-1.5})^2/1.5^2$,
i.e.\ approximately $200$. A realisation with 220 points is shown in the top-left panel of 
Figure~\ref{RealisationPoi}. The temporal progress of the process is 
illustrated in the top-right panel of Figure~\ref{RealisationPoi}, which 
shows the cumulative number of points as a function of time, i.e.\ 
$N(t)=Y(W_S\times[0,t])$, $t\in[0,1]$. The lower row of Figure~\ref{RealisationPoi} shows two 
spatial projections. In the left panel, we display $Y\cap(W_S\times[0,0.5])$, 
in the right panel $Y\cap(W_S\times[0.5,1])$. 
Here the decline in intensity, with increasing $y$-coordinate, is clearly visible. 
Furthermore, a comparison of the two spatial projections illustrates the exponential decay in the intensity function.

\begin{figure}[htbp]
\begin{center}
\begin{tabular}{cc}
     	{\includegraphics[width=0.35\textwidth]{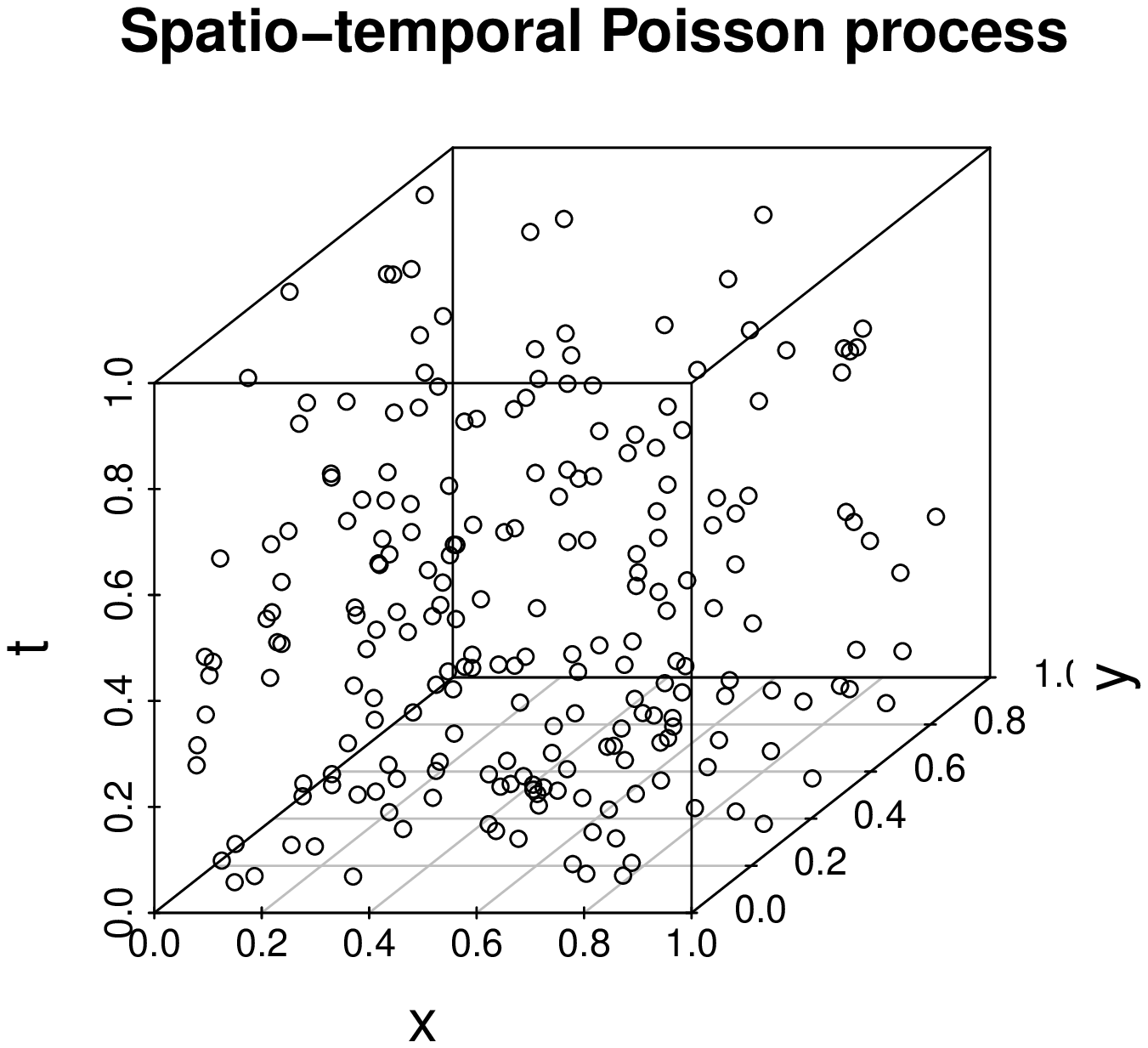}}
	&
     	{\includegraphics[width=0.35\textwidth]{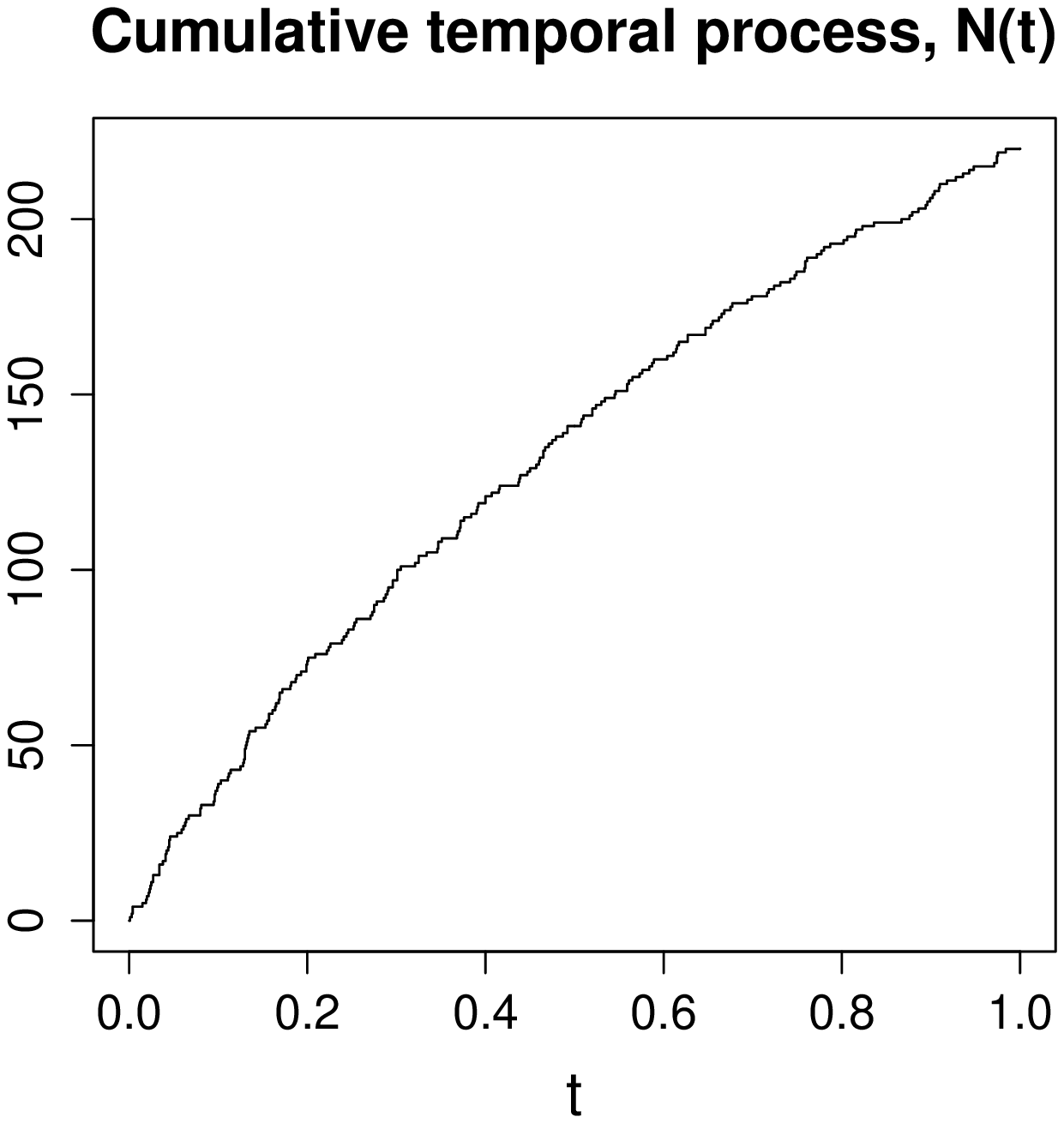}}
	\\
	{\includegraphics[width=0.35\textwidth]{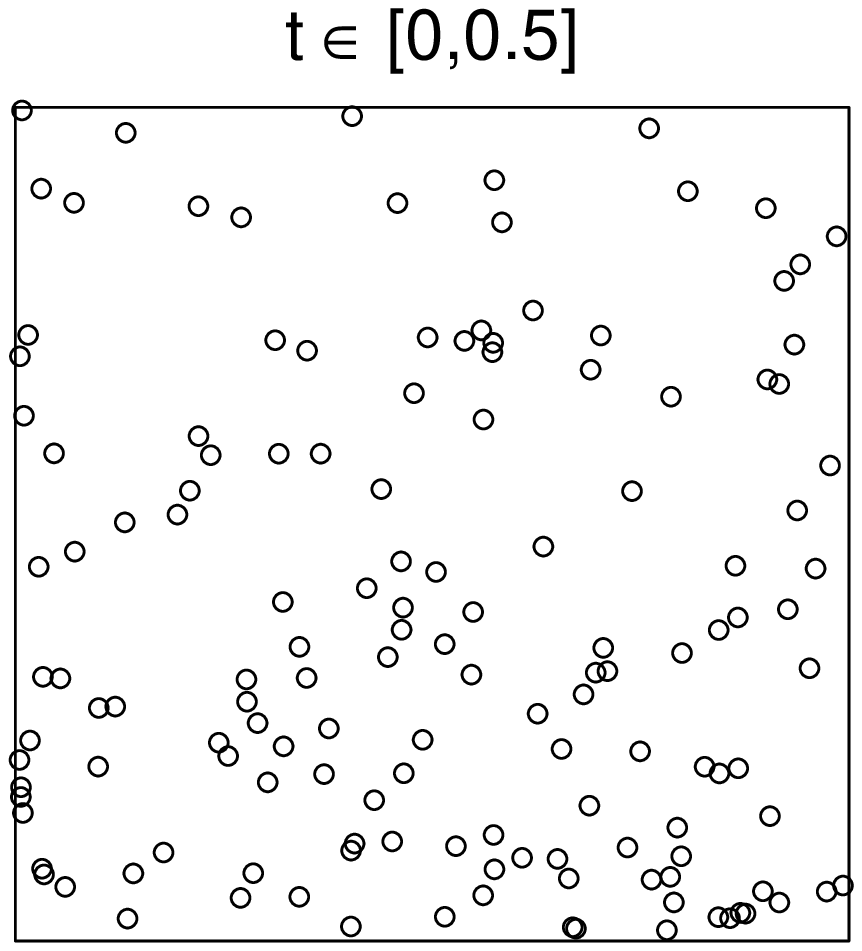}}
	&
     	{\includegraphics[width=0.35\textwidth]{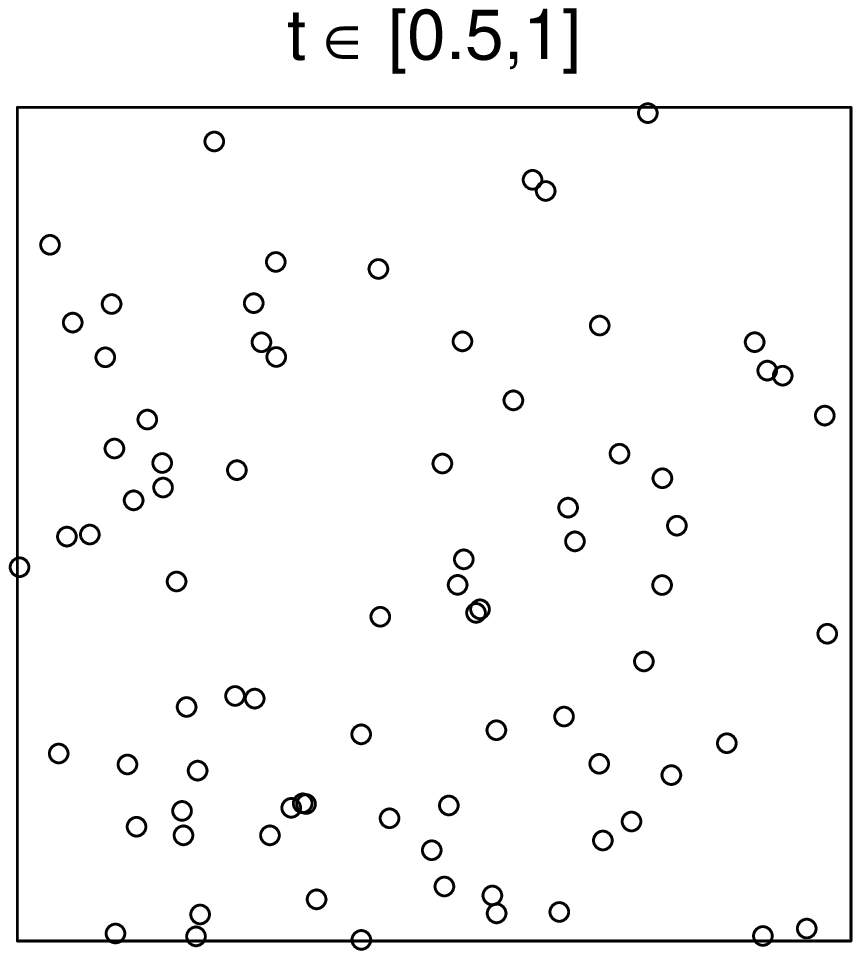}}
\end{tabular}
\caption{A realisation on $W_S\times W_T=[0,1]^2\times[0,1]$ of a Poisson 
process with intensity function $\lambda((x,y),t) = 750\e^{-1.5(y+t)}$, $x,y,t\in\R$.
Upper row: A 3-d plot (left) and a plot of the associated cumulative count 
process (right). Lower row: Spatial projections for the time
intervals $[0,0.5]$ (left) and $[0.5,1]$ (right).
}
\label{RealisationPoi}
\end{center}
\end{figure}

In Figure \ref{PoiFGest}, on the left, we show a collection of 2-d plots of the estimates of 
$G_{\rm inhom}(r,t_0)$ and $F_{\rm inhom}(r,t_0)$ for a 
fixed set of values $t=t_0$. Similarly, in the rightmost plot, 
we display a collection of 2-d plots of the estimates for a fixed set of values $r=r_0$. 
In both cases the dotted lines (-$\blacklozenge$-) represent 
the estimated empty space functions. 
We see that throughout the two estimates are approximately equal and, in addition, there are instances where each of the two is larger than the other. 

\begin{figure}[htbp]
\begin{center}
\begin{tabular}{cc}
     	{\includegraphics[width=0.45\textwidth]{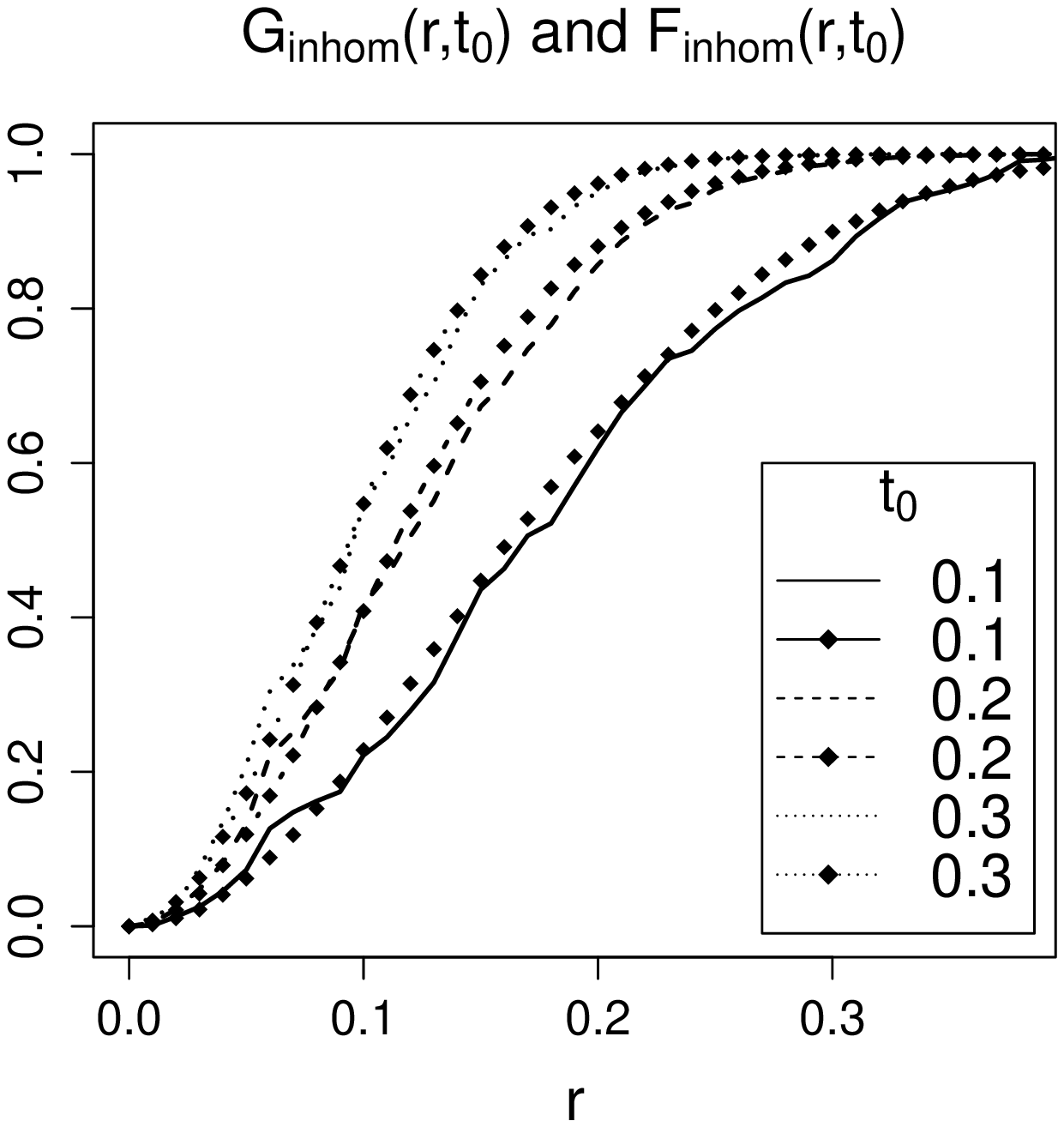}}
	&
     	{\includegraphics[width=0.45\textwidth]{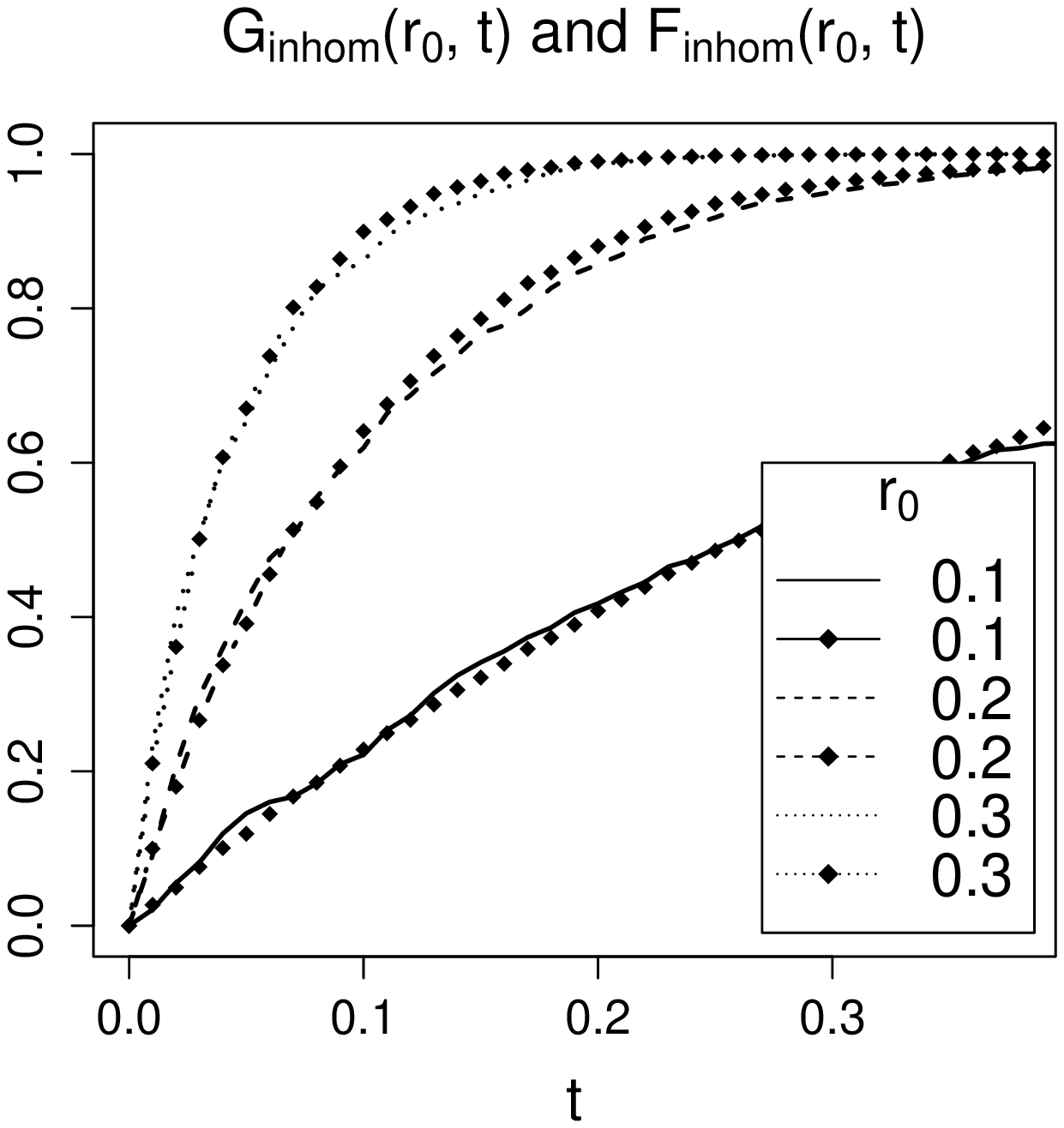}}
\end{tabular}
\caption{
Plots of the estimated nearest neighbour distance distribution
function and empty space function of the Poisson process sample in Figure \ref{RealisationPoi}. Left: As a function of spatial
distance for fixed temporal distances $t_0$. 
Right: As a function of temporal distance for fixed spatial
distances $r_0$. In both plots, the dotted lines 
(-$\blacklozenge$-) represent the empty space function.
}
\label{PoiFGest}
\end{center}
\end{figure}

\subsection{Thinned hard core process}

Let $Y$ be the location dependent thinning of a stationary spatio-temporal
hard core process as described in Section \ref{SectionThinning} with
retention probability $p(x,y,t)=\e^{-1.5(y+t)}$, $(x,y,t)\in\R^2\times\R$,
$\beta=1300$, $R_S=0.05$ and $R_T=0.05$. The associated realisations are shown in the 
top panels of Figure~\ref{RealisationStrauss}. 
The underlying hard core process has 762 points, whereby $\hat{\lambda} = 762$, and the thinned process has 204 points. 
Note that  
the expected number of observed points of $Y$ in $W_S\times W_T$ is
$\lambda\int_{[0,1]^3}p(x,y,t)d(x,y,t)\approx\hat{\lambda}(1-\e^{-1.5})^2/1.5^2$, i.e.\ approximately $200$.
The temporal progress of the process is illustrated in the top-right panel of 
Figure~\ref{RealisationStrauss}, which shows the cumulative number of points 
as a function of time, i.e.\ $N(t)=Y(W_S\times[0,t])$, $t\in[0,1]$. 
The lower row of Figure~\ref{RealisationStrauss} shows two spatial projections. In the lower left panel 
we display $Y\cap(W_S\times[0,0.5])$ and in the lower right panel $Y\cap(W_S\times[0.5,1])$.
Just as in the Poisson case, the two spatial projections illustrate the decay in the intensity function, both in the $y$- and $t$-dimensions.

\begin{figure}[htbp]
\begin{center}
\begin{tabular}{ccc}
     	{\includegraphics[width=0.35\textwidth]{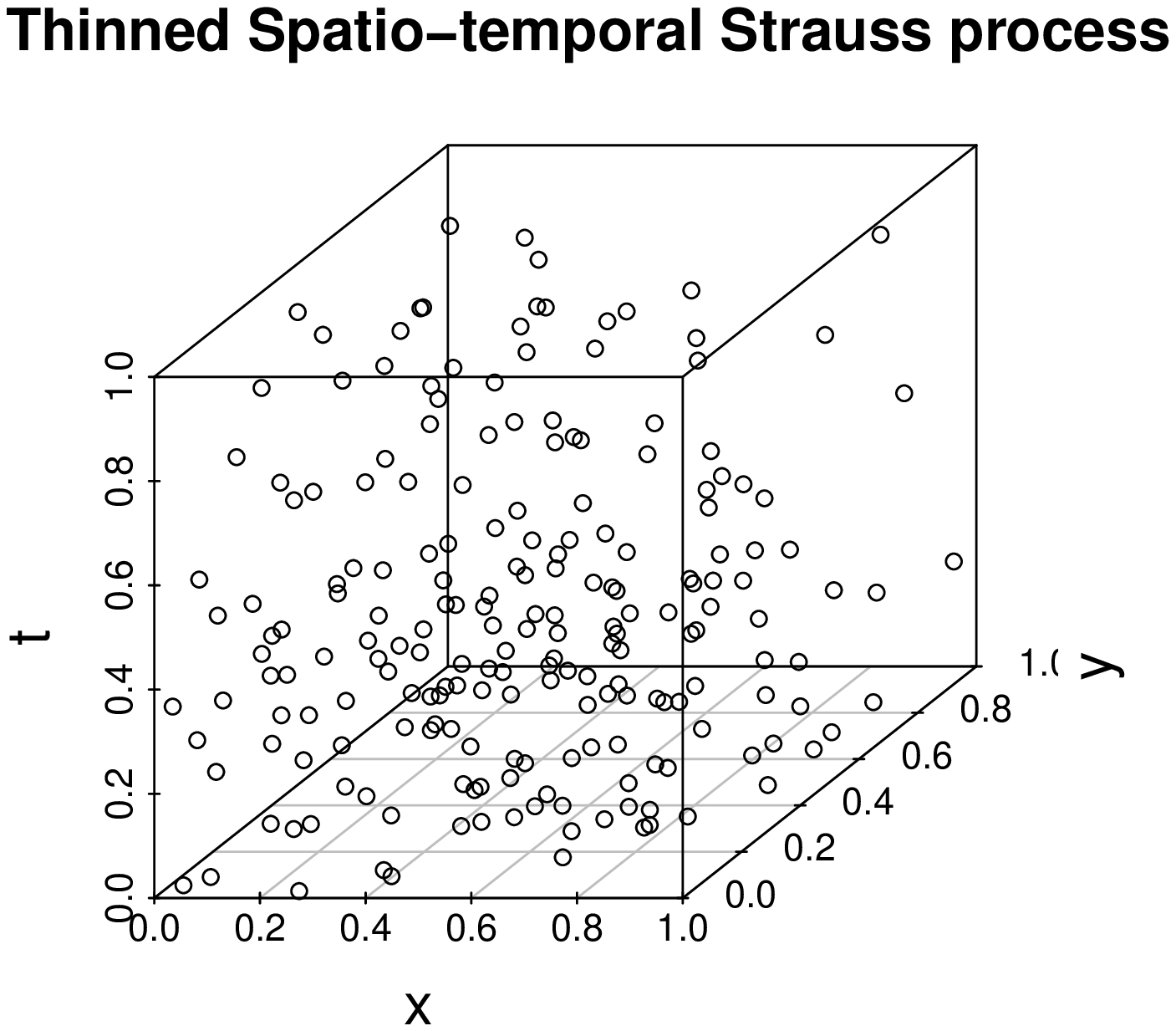}}
	&
     	{\includegraphics[width=0.35\textwidth]{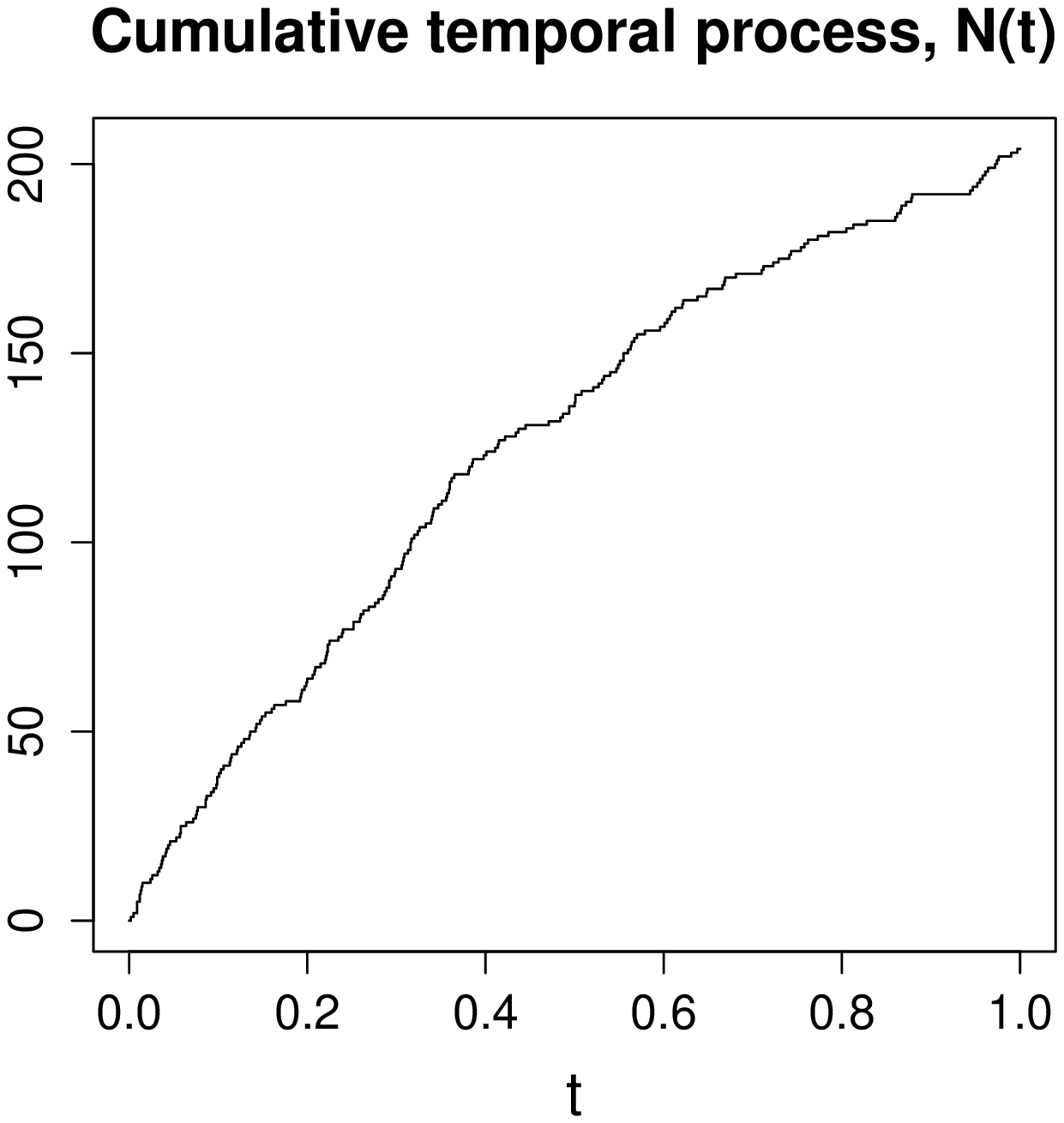}}
	\\
	{\includegraphics[width=0.35\textwidth]{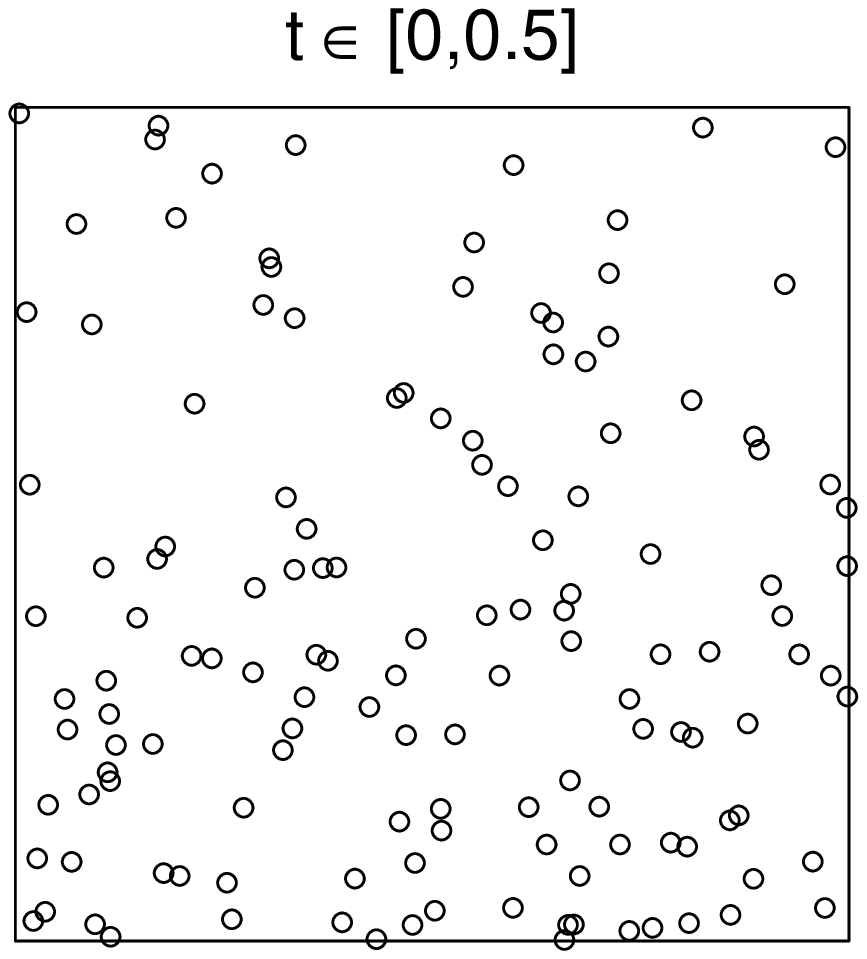}}
	&
     	{\includegraphics[width=0.35\textwidth]{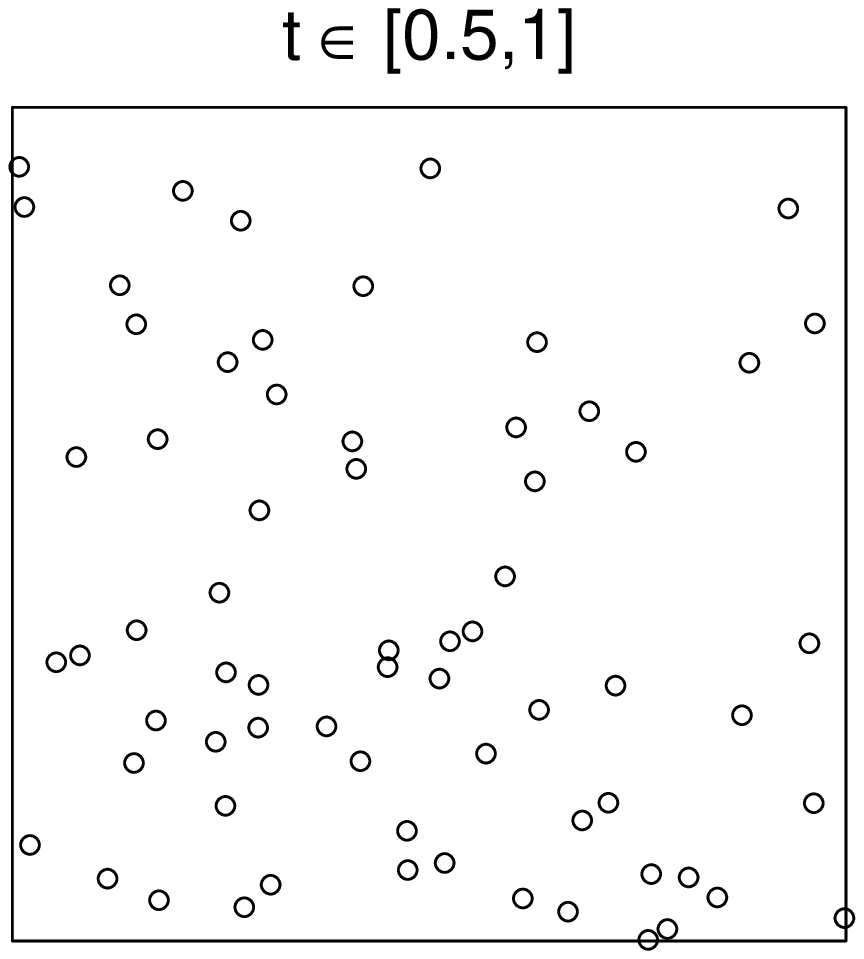}}
\end{tabular}
\caption{A realisation on $W_S\times W_T=[0,1]^2\times[0,1]$ of a 
location dependent thinning, with retention probability $p(x,y,t)=\e^{-1.5(y+t)}$, $(x,y,t)\in\R^2\times\R$, of a spatio-temporal hard core process with 
$\beta=1300$, $R_S=0.05$ and $R_T=0.05$. 
Upper row: A 3-d plot (left) and a plot of the associated cumulative count 
process (right). Lower row: Spatial projections for the time
intervals $[0,0.5]$ (left) and $[0.5,1]$ (right).
}
\label{RealisationStrauss}
\end{center}
\end{figure}

In order to obtain a picture of the interaction structure, in Figure \ref{StraussFGest} we have plotted the estimates of $G_{\rm inhom}(r,t)$ and $F_{\rm inhom}(r,t)$ for large $r,t$ ranges. 
Again, the dotted lines (-$\blacklozenge$-) represent the estimated empty space functions. 
As expected we see that the estimate of $G_{\rm inhom}(r,t)$ is the smaller of the two. 
For small values of $r_0$ and $t_0$, the hardcore distances $R_S$ and $R_T$ are clearly identified 
in the lower row of Figure \ref{StraussFGest}. 
For large values of $r_0$ and $t_0$ this is no longer the case due to accumulation points; see the top row of Figure \ref{StraussFGest}. For instance, there are many points violating the spatial hard core constraint, which still respect the temporal temporal hard core constraint. 

\begin{figure}[htbp]
\begin{center}
\begin{tabular}{cc}
     	{\includegraphics[width=0.45\textwidth]{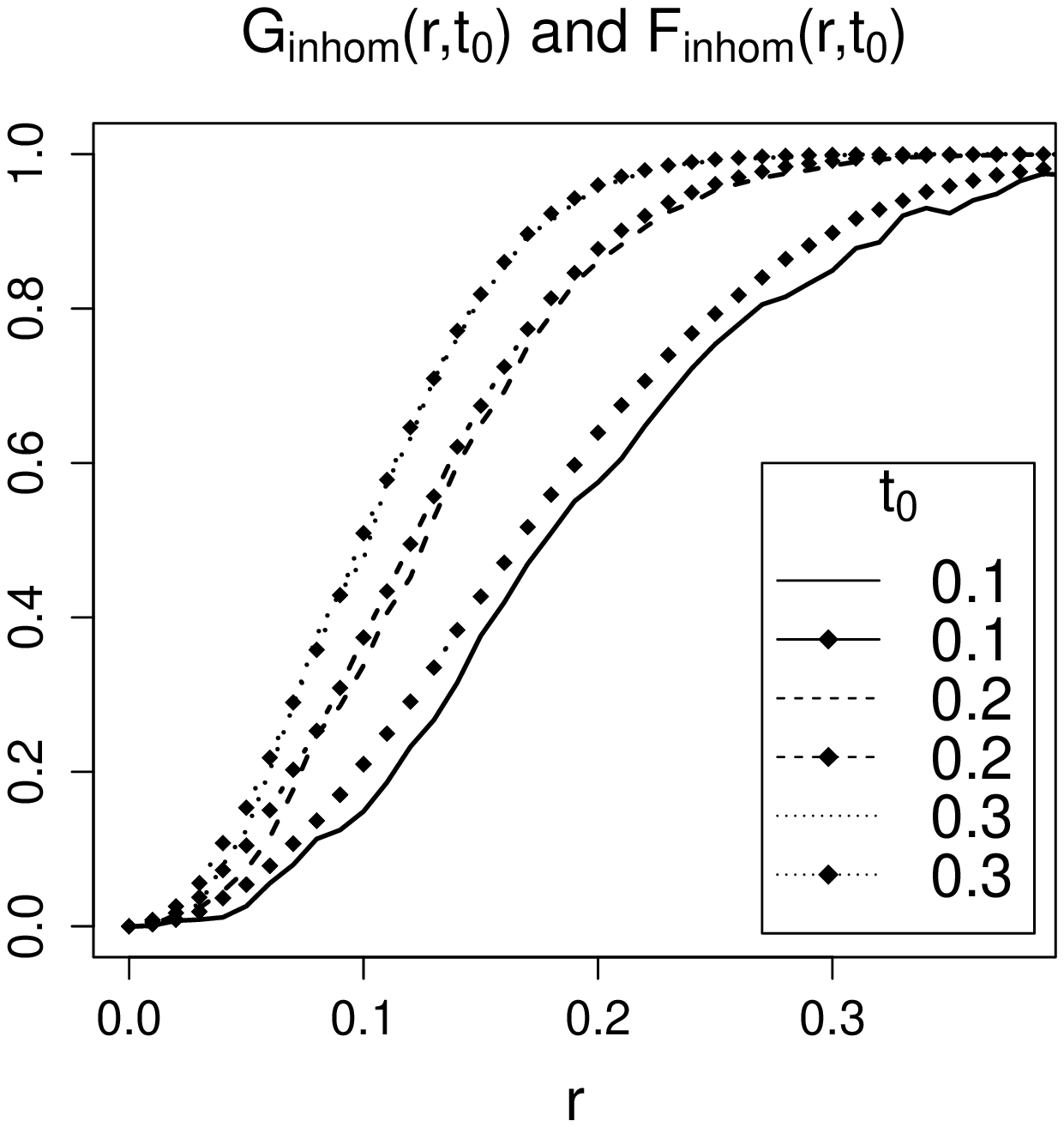}}
	&
     	{\includegraphics[width=0.45\textwidth]{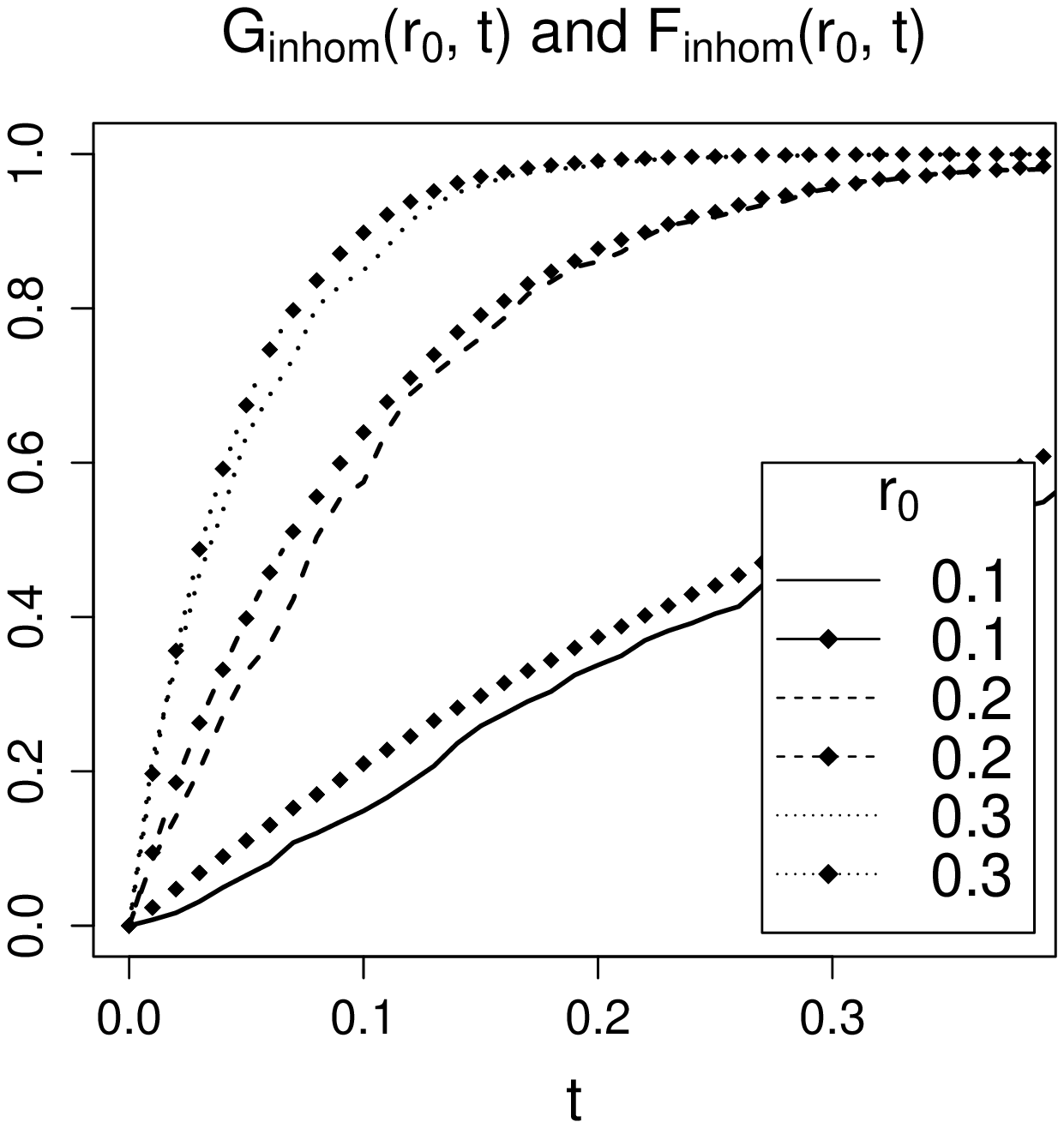}}
	\\
	{\includegraphics[width=0.45\textwidth]{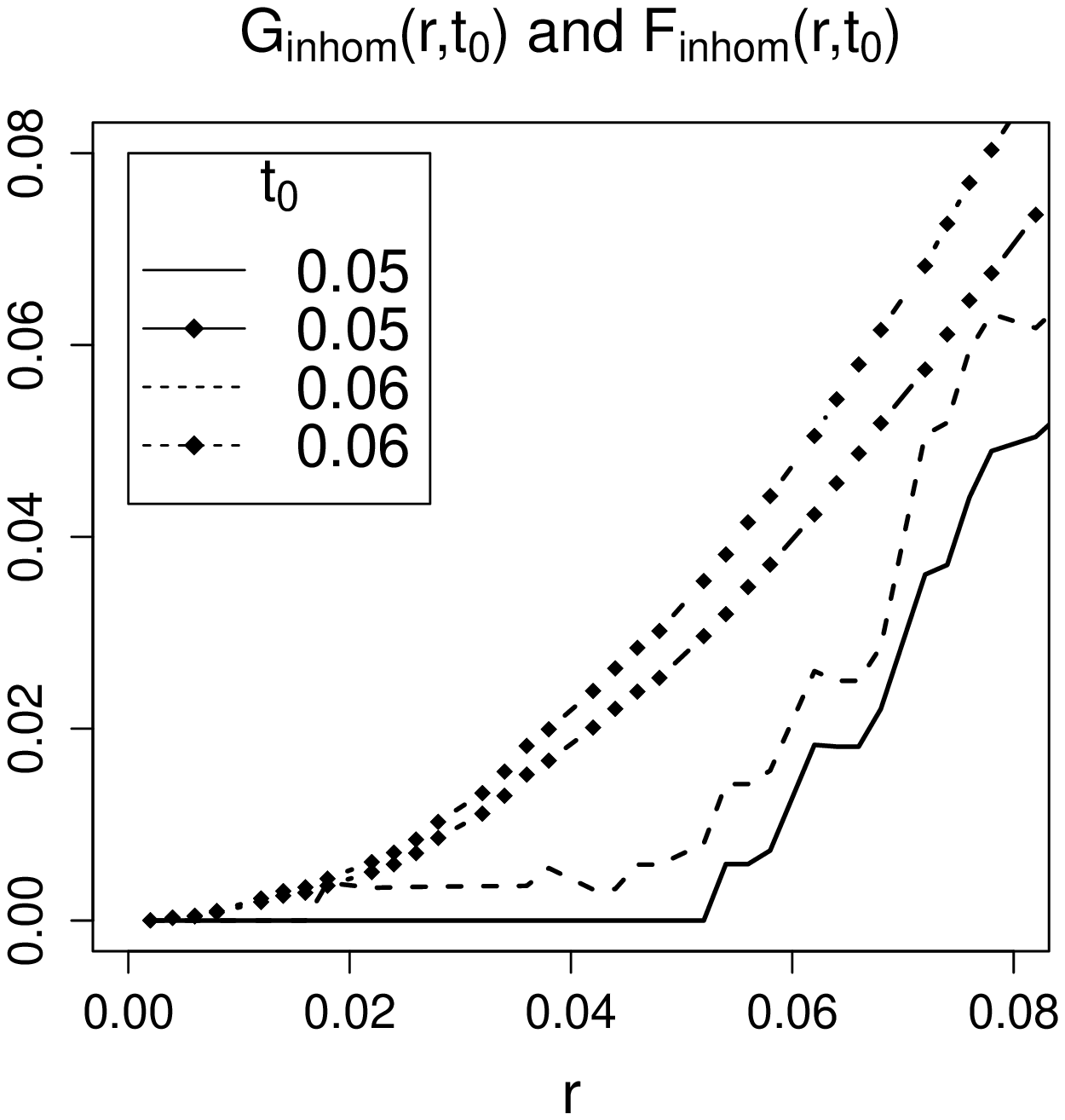}}
	&
     	{\includegraphics[width=0.45\textwidth]{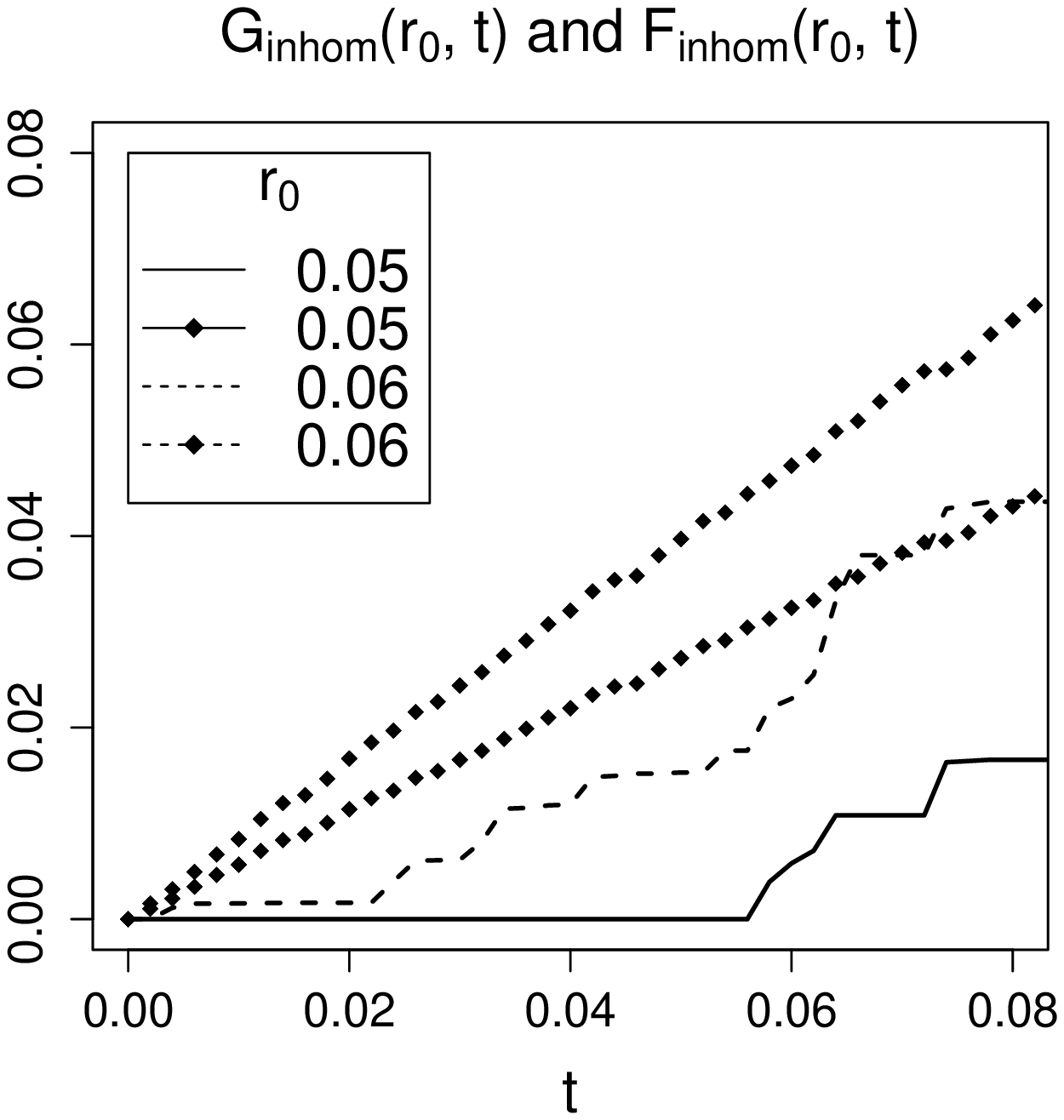}}
\end{tabular}
\caption{
Plots of the estimated nearest neighbour distance distribution
function and the empty space function of the thinned hard core process in Figure \ref{RealisationStrauss}. Left: As a function of spatial
distance for fixed temporal distances $t_0$. 
Right: As a function of temporal distance for fixed spatial
distances $r_0$. In both cases the dotted lines 
(-$\blacklozenge$-) represent the empty space function estimates.
}
\label{StraussFGest}
\end{center}
\end{figure}

\subsection{Log-Gaussian Cox process}

Recall the log-Gaussian Cox processes discussed in Section \ref{SectionCox}.
We consider the separable covariance 
function (see Appendix)
\[
C((x_1,y_1,t_1),(x_2,y_2,t_2)) = C_S((x_1,y_1)-(x_2,y_2)) C_T(t_1-t_2)
\]
for the driving Gaussian random field of the STPP $Y$. 
Specifically, we let the component covariance functions be $C_S(x,y) = \sigma_S^2\exp\{- \|(x,y)\|^2\}$ 
(Gaussian) and $C_T(t) = \sigma_T^2\exp\{-|t|\}$ (exponential), $x,y,t\in\R$, where $\sigma_S^2=\sigma_T^2=1/4$, so that 
$\sigma^2 = C_S(0)C_T(0) = 1/16$ and we let the mean function be given by $\mu(x,y,t)= \log(750) - 1.5(y + t) - \sigma^2/2$. 
Figure~\ref{IntensityCox} shows projections of a realisation of the 
driving random intensity function at time $t=0.5$ (left) and spatial
coordinate $x=0.5$ (right). 
Note the gradient in the vertical direction in the left plot and the diagonal trend in the right one. 

\begin{figure}[bhtp]
\begin{center}
\begin{tabular}{cc}
     	{\includegraphics[width=0.45\textwidth]{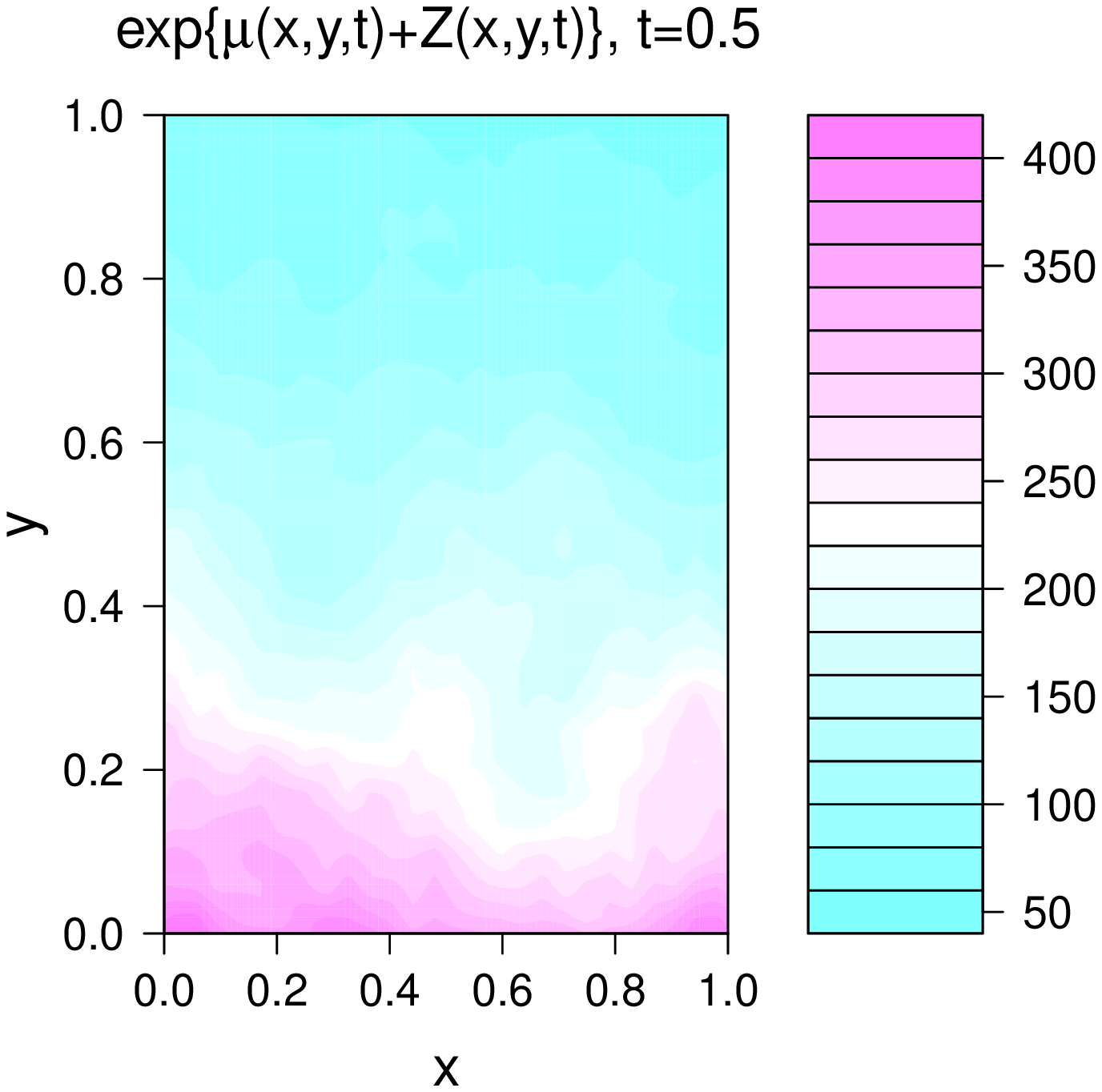}}
	&
     	{\includegraphics[width=0.45\textwidth]{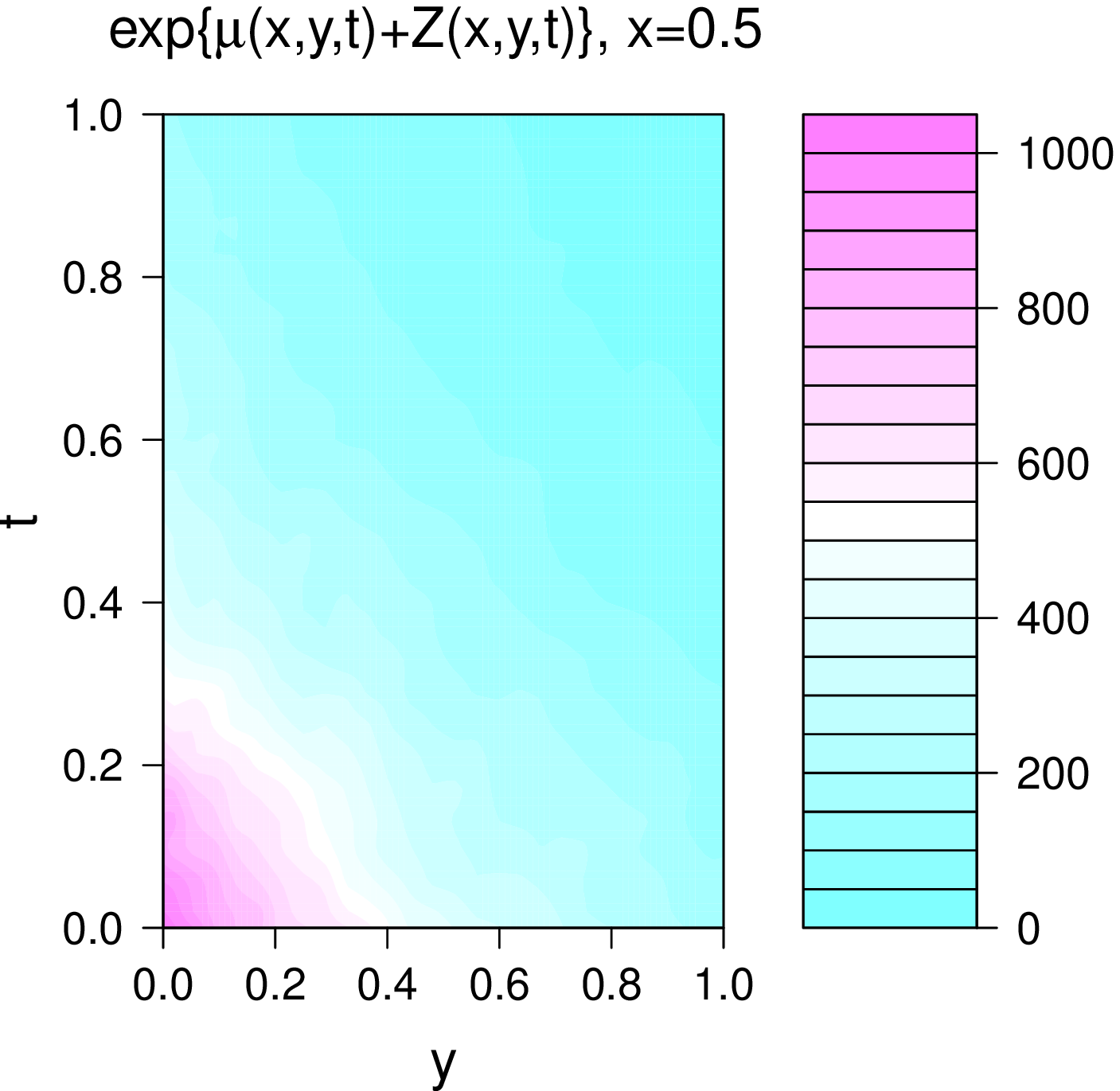}}
\end{tabular}
\caption{Projections of a realisation of the driving random intensity field at time $t=0.5$ 
(left) and at spatial coordinate $x=0.5$ (right). 
}
\label{IntensityCox}
\end{center}
\end{figure}

By expression (\ref{IntesityLGCP}) we obtain $\lambda(x,y,t)=750\e^{-1.5(y+t)}$ and consequently 
$\bar{\lambda} = 750\e^{-3}\approx37.34$. Hereby the expected number of observed points of $Y$ in $W_S\times W_T$ is
$750(1 - e^{-1.5})^2/1.5^2$, i.e.\ approximately $200$. 
A realisation of $Y$ with 219 points is shown in the 
top-left panel of Figure~\ref{RealisationCox} and 
the cumulative number of points as a function of time, 
i.e.\ $N(t)=Y(W_S\times[0,t])$, $t\in[0,1]$, is illustrated in the top-right 
panel. 
The lower row of Figure~\ref{RealisationCox} shows two spatial projections. In the left 
panel, we display $Y\cap(W_S\times[0,0.5])$ and in the right panel 
$Y\cap(W_S\times[0.5,1])$. 
Also here the decay in the intensity function in the $y$- and $t$-dimensions is visible. In addition, by comparing the lower rows of Figures \ref{RealisationPoi} and \ref{RealisationCox}, the present clustering effects become evident. 

\begin{figure}[htbp]
\begin{center}
\begin{tabular}{cc}
     	{\includegraphics[width=0.35\textwidth]{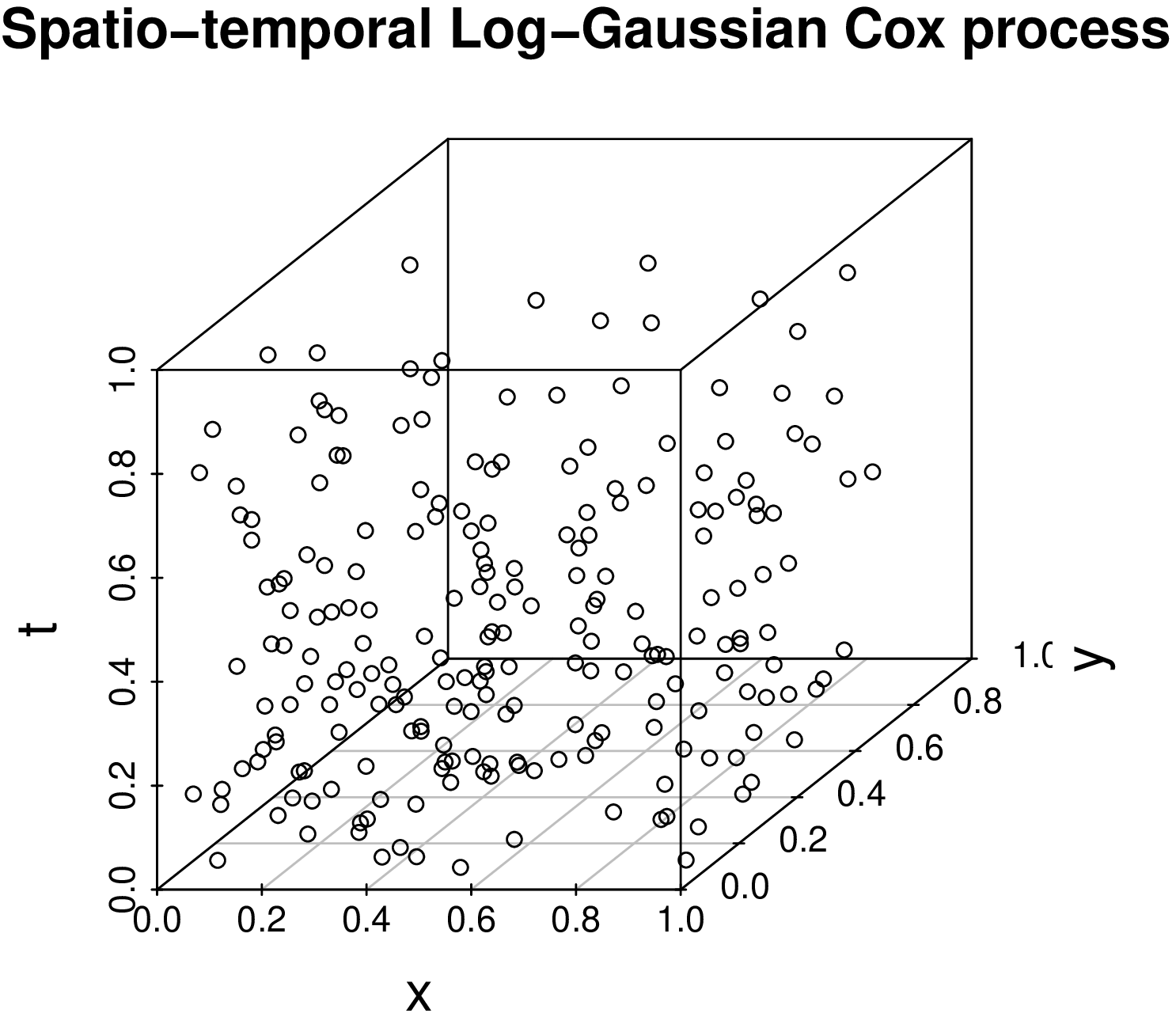}}
	&
    	{\includegraphics[width=0.35\textwidth]{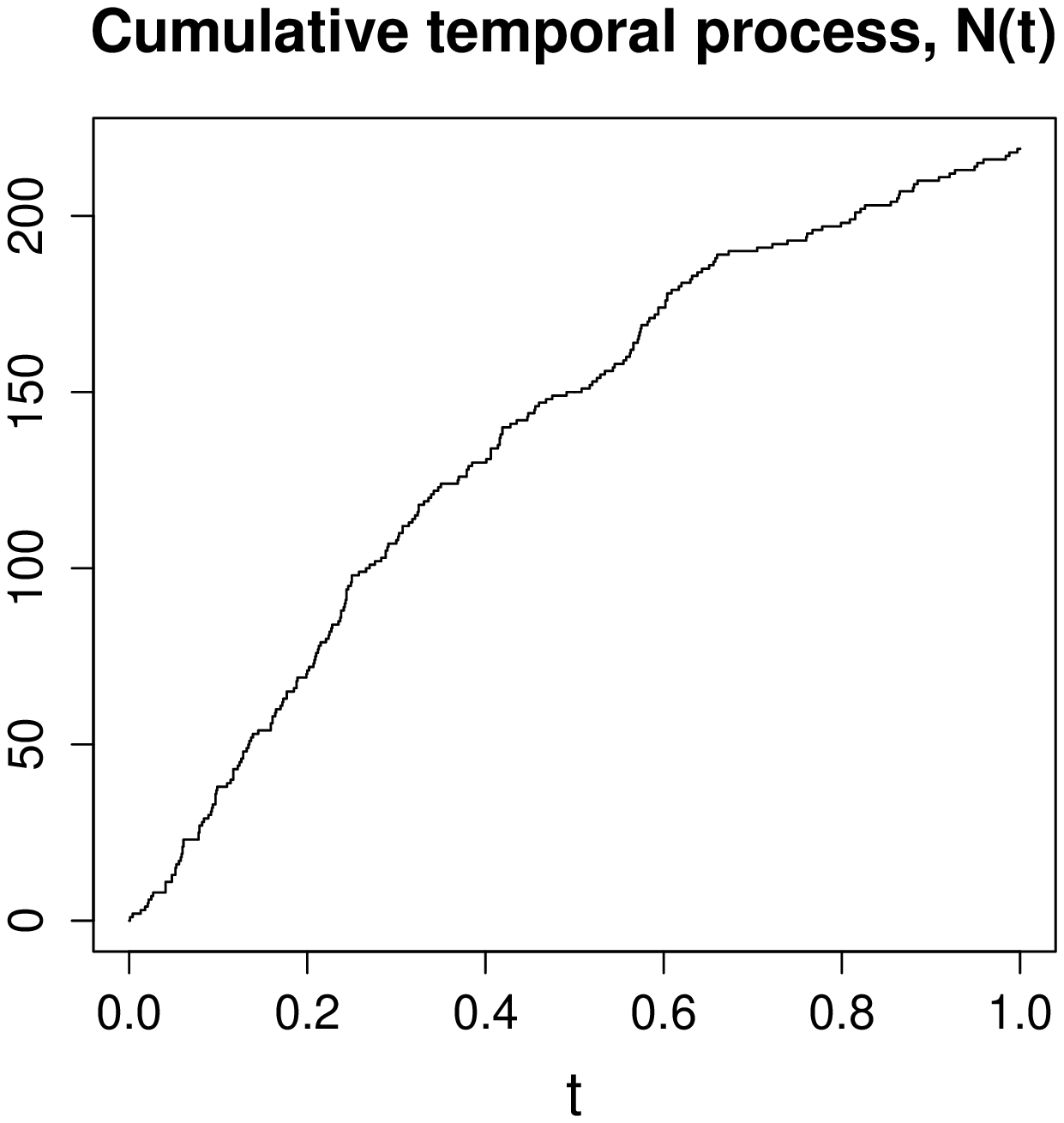}}
	\\
     	{\includegraphics[width=0.35\textwidth]{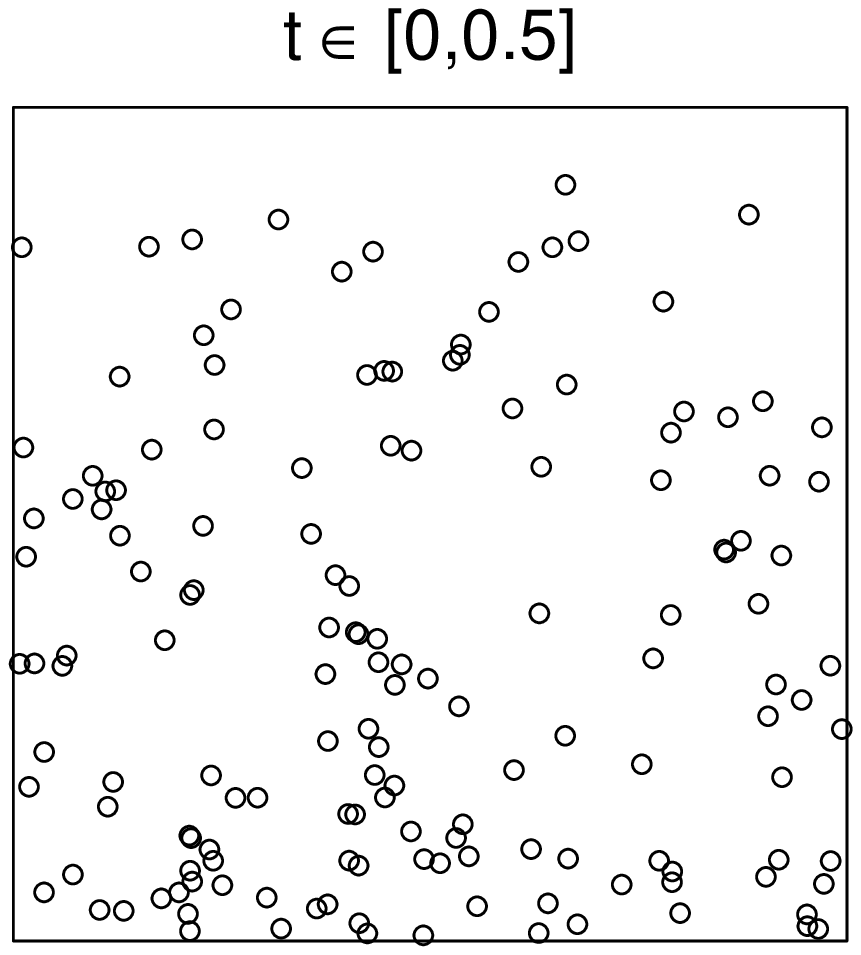}}
	&
	{\includegraphics[width=0.35\textwidth]{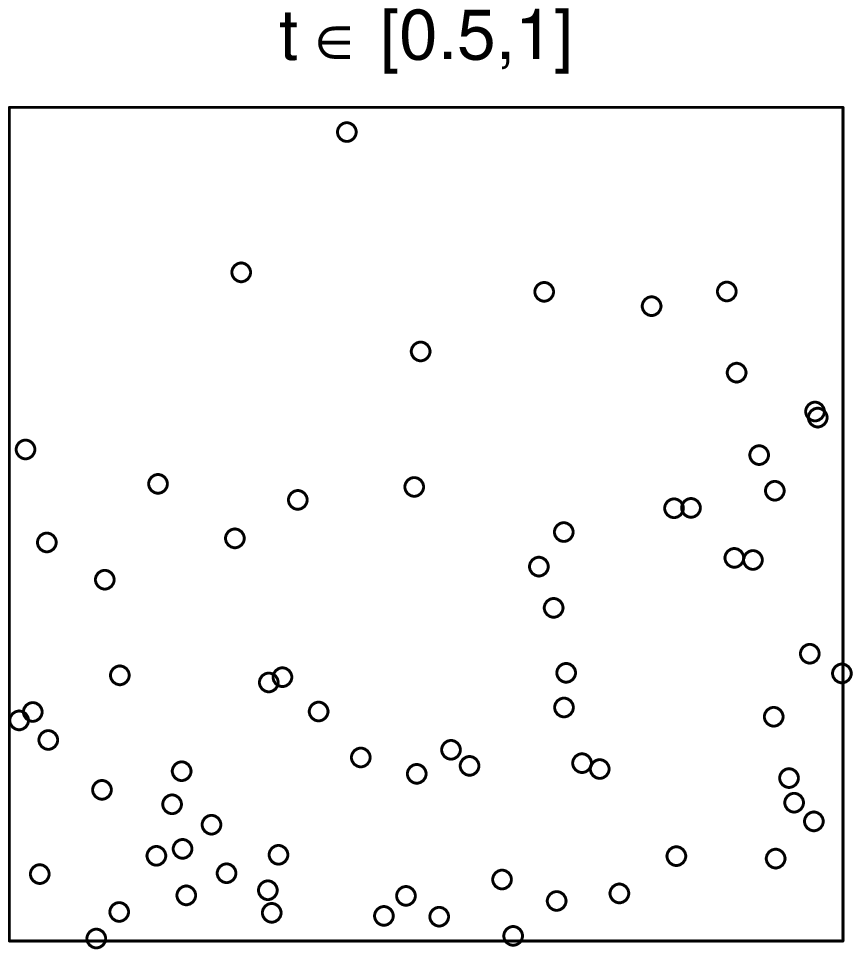}}
\end{tabular}
\caption{A realisation on $W_S\times W_T=[0,1]^2\times[0,1]$ of a 
log-Gaussian Cox process with $\mu(x,y,t)= \log(750) - 1.5(y + t) - \frac{(1/4)^2}{2}$ 
and separable covariance function 
$C_S(x,y)C_T(t) = \frac{1}{4}\e^{-\|(x,y)\|^2} \frac{1}{4}\e^{-|t|}$, 
$(x,y,t)\in\R^2\times\R$. 
Upper row: A 3-d plot (left) and a plot of the associated cumulative count 
process (right). Lower row: Spatial projections for the time 
intervals $[0,0.5]$ (left) and $[0.5,1]$ (right).
}
\label{RealisationCox}
\end{center}
\end{figure}

In Figure~\ref{CoxFGest} we have plotted the estimates of $G_{\rm inhom}(r,t)$ and $F_{\rm inhom}(r,t)$. As before, the dotted lines (-$\blacklozenge$-) represent the estimates of $F_{\rm inhom}(r,t)$. Due to the structure of $C$, when e.g.\ $t$ is small we find clear signs of clustering whereas for larger $t$, where $C(x,y,t)\approx0$, as expected we have Poisson like behaviour. 

\begin{figure}[htbp]
\begin{center}
\begin{tabular}{cc}
     	{\includegraphics[width=0.45\textwidth]{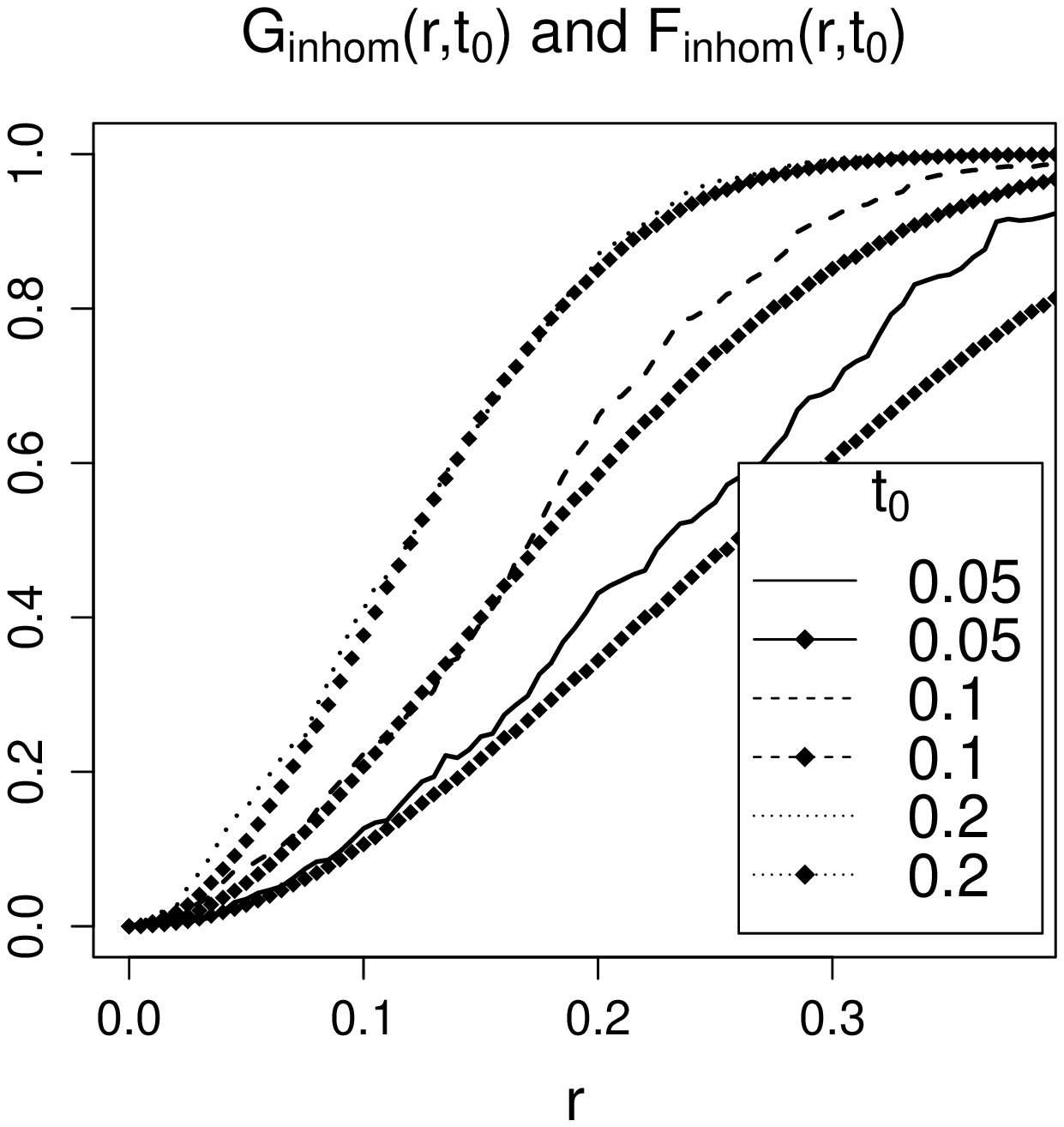}}
	&
     	{\includegraphics[width=0.45\textwidth]{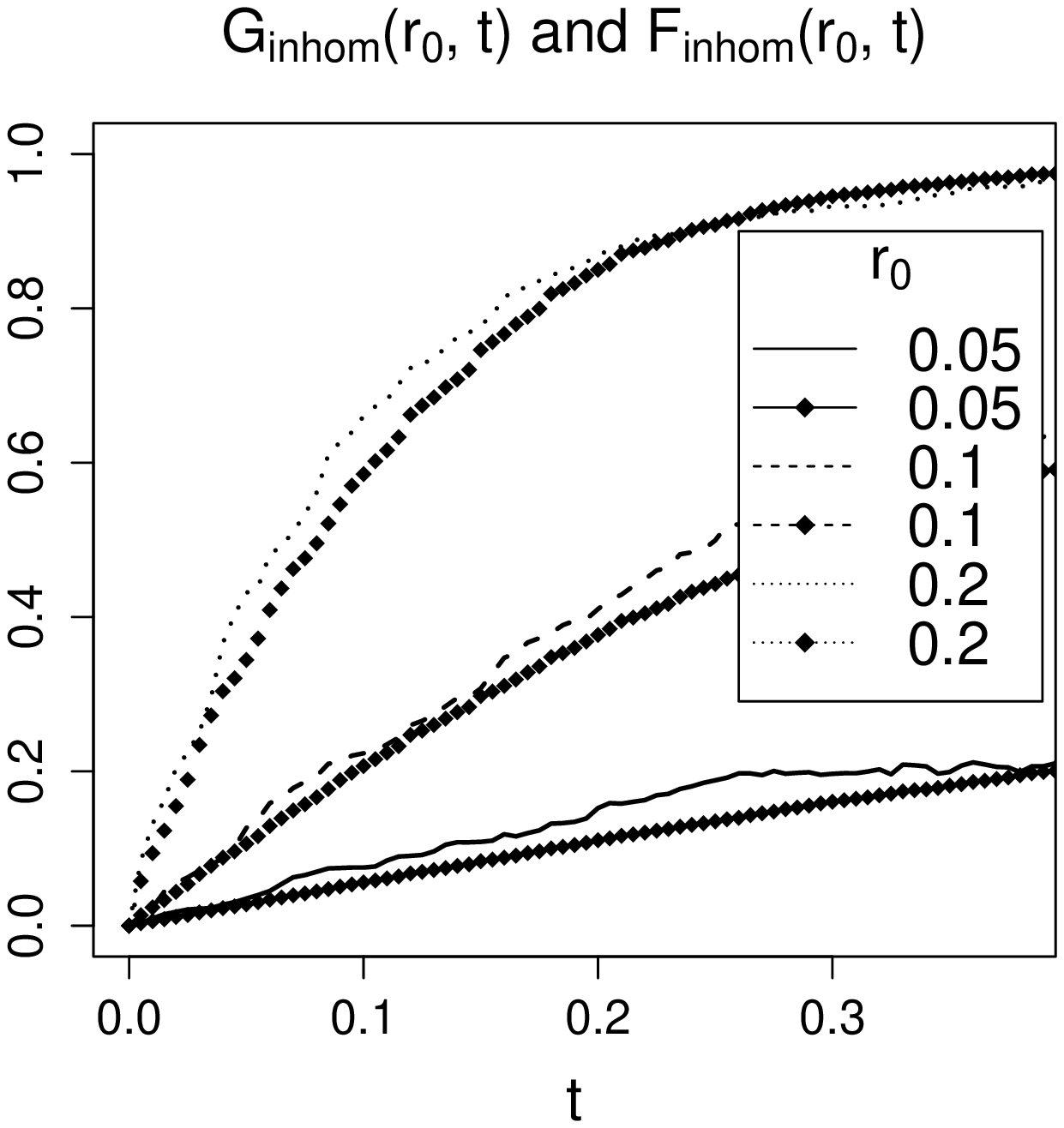}}
\end{tabular}
\caption{
Plots of the estimated nearest neighbour distance distribution
function and empty space function of the log-Gaussian Cox process in Figure \ref{RealisationCox}. Left: As a function of spatial
distance for fixed temporal distances $t_0$. 
Right: As a function of temporal distance for fixed spatial
distances $r_0$. In both plots, the dotted lines 
(-$\blacklozenge$-) represent the empty space function.
}
\label{CoxFGest}
\end{center}
\end{figure}

\section*{Acknowledgements}
The authors would like to thank Guido Legemaate, Jesper M{\o}ller and 
Alfred Stein for useful input and discussions, Martin Schlather 
for the providing of and the help with an updated version of his R 
package {\tt{RandomFields}}. This research was supported by the 
Netherlands Organisation for Scientific Research NWO (613.000.809).


\newpage

\section*{Appendix}

\subsection*{Sample path continuity of Gaussian random fields}

Let $Z$ be a stationary Gaussian random field with mean zero. We wish to impose 
conditions, which ensure that $Z$ a.s.\ has continuous sample paths. 
If $Z$ would be defined on the Euclidean space 
$(\R^{d}\times\R,\|\cdot\|_{\R^{d+1}},d_{\R^{d+1}}(\cdot,\cdot))$, 
with $C(x,y) = \sigma^2 r( x-y )$, $\sigma^2 > 0$,
then \cite[Section 5.6.1]{MollerWaagepetersen} lists sufficient conditions on the 
correlation function $r(\cdot)$ as follows.
There exist $\epsilon,\delta>0$ such that either, 
\begin{enumerate}
\item $1-r(x,t)<\delta/(-\log(\|(x,t)\|_{\R^{d+1}}))^{1+\epsilon}$, or 
\item $1-r(x,t)<\delta\|(x,t)\|_{\R^{d+1}}^{\epsilon}$,
\end{enumerate}
for all lag pairs $(x,t)\in\R^{d}\times\R$ in an open Euclidean ball
centred at 0. Note that the former condition, which in fact is the 
condition given in \cite[Theorem 3.4.1]{Adler}, is less restrictive than 
the latter one but often harder to check.
However, the underlying space here is 
$(\R^{d}\times\R,\|\cdot\|_{\infty},d(\cdot,\cdot))$. 
Hence, one explicit way of obtaining equivalent conditions for 
$r(\cdot)$ would be to consider the log-entropy related results of 
\cite[Section 1]{AdlerTaylor} for Gaussian random fields on general 
compact spaces and exploit that $\R^{d}\times\R$ is $\sigma$-compact. 
A more direct and natural approach is to note that, through the
topological equivalence of $d_{\R^{d+1}}(\cdot,\cdot)$ and $d(\cdot,\cdot)$, 
we have the necessary condition that, for any $(x,t)\in\R^{d}\times\R$, 
there exist constants $\alpha_1,\alpha_2>0$ such that 
$\alpha_1 d((x,t),(y,s)) \leq d_{\R^{d+1}}((x,t),(y,s)) 
\leq \alpha_2 d((x,t),(y,s))$ for all $(y,s)\in\R^{d}\times\R$. 
Hereby, in particular, there are $\alpha_1,\alpha_2>0$ such that 
$\alpha_1 \|(x,t)\|_{\infty} \leq \|(x,t)\|_{\R^{d+1}} \leq 
\alpha_2 \|(x,t)\|_{\infty}$ for all $(x,t)\in\R^{d}\times\R$ and we see 
that the conditions above are retained in 
$(\R^{d}\times\R,\|\cdot\|_{\infty},d(\cdot,\cdot))$, with adjusted 
constants $\delta,\epsilon>0$. 
Note that the a.s.\ sample path continuity implies a.s.\ sample 
path boundedness on compact sets \cite[Section 1]{AdlerTaylor}. 

\subsection*{Covariance models}

One particular family of correlation functions $r(\cdot)$ for which the 
a.s.\ continuity conditions above are satisfied is the {\em power 
exponential family\/} (see \cite[Section 5.6.1]{MollerWaagepetersen}),
$$
r(x,t) = \exp(-\|(x,t)\|_{\infty}^{\delta}),
\quad 0\leq\delta\leq2,
\quad (x,t)\in\R^{d}\times\R. 
$$
The special case $\delta=1$ generates the exponential correlation, 
$\delta=2$ gives rise to the Gaussian correlation function. 
Note that the isotropy of $r(\cdot)$ implies isotropy of the LGCP $Y$ 
since its distribution is completely specified by $C(\cdot)$.

A common practical assumption when modelling spatio-temporal Gaussian random 
fields is to assume separability (see e.g.\ \cite[Chapter~23]{Handbook}). 
Consider the covariance functions 
$C_S(x,y) = \sigma_S^2 r_S(x-y)$ 
and 
$C_T(t,s) = \sigma_T^2 r_T(t-s)$, $x,y\in\R^d$, $t,s\in\R$, 
where $\sigma^2_S,\sigma^2_T>0$.
We may now consider two types of separability:
\begin{enumerate}
\item Multiplicative separability: 
\[
C((x,t),(y,s))=C_S(x,y)C_T(t,s)=\sigma_S^2 \sigma_T^2 r_S(x-y) r_T(t-s).
\]
\item Additive separability: 
\[
C((x,t),(y,s))=C_S(x,y) + C_T(t,s)=\sigma_S^2 r_S(x-y) + \sigma_T^2 r_T(t-s). 
\]
\end{enumerate}
The latter is a consequence of assuming that
 $Z(x,t)=Z_S(x) + Z_T(t)$, $(x,t)\in\R^{d}\times\R$, where $Z_S(x)$ and $Z_T(t)$ are independent mean zero Gaussian random fields with covariance functions 
$C_S$ and $C_T$, respectively. 
In both cases a separable power exponential model can be obtained by letting 
$r_S(x) = \exp(-\|x\|^{\delta_S})$ and $r_T(t) = \exp(-|t|^{\delta_T})$ for $\delta_S,\delta_T\in[0,2]$. 


\begin{thebibliography}{99}

\bibitem{Adler}
Adler, R.J. (1981).
{\em The Geometry of Random Fields.}
Wiley.

\bibitem{AdlerTaylor}
Adler, R.J., Taylor, J.E. (2007). 
{\em Random Fields and Geometry.}
Springer (Monographs in Mathematics). 

\bibitem{BaddeleyEtAl}
Baddeley, A.J., M{\o}ller, J., Waagepetersen, R. (2000). 
Non- and semi-parametric estimation of interaction in 
inhomogeneous point patterns. 
{\em Statistica Neerlandica} 54, 329--350.

\bibitem{BaddeleyTurner}
Baddeley, A., Turner, R. (2005).
Spatstat: An {\bf R} package for analyzing spatial point patterns. 
{\em Journal of Statistical Software} 12, 1--42.

\bibitem{Bedford}
Bedford, T., Van den Berg, J. (1997). 
A remark on the Van Lieshout and Baddeley $J$-function for point processes. 
{\em Advances in Applied Probability} 29, 19--25.

\bibitem{Brix}
Brix, A., Diggle, P.J. (2001). 
Spatiotemporal prediction for Log-Gaussian Cox Processes.
{\em Journal of the Royal Statistical Society. Series B 
(Statistical Methodology)} 63, 823--84.

\bibitem{SKM}
Chiu, S.N., Stoyan, D., Kendall, W. S., Mecke, J. (2013).
{\em Stochastic Geometry and its Applications.}
Third Edition. 
Wiley.

\bibitem{ColeJone91}
Coles, P., Jones, B. (1991).
A lognormal model for the cosmological mass distribution.
{\em Monthly Notices of the Royal Astronomical Society} 248, 1--13.

\bibitem{CronieSarkka}
Cronie, O., S\"arkk\"a, A. (2011). 
Some edge correction methods for marked spatio-temporal point process models. 
{\em Computational Statistics \& Data Analysis} 55, 2209--2220. 

\bibitem{DVJ1}
Daley, D.J., Vere-Jones, D. (2003).
{\em An Introduction to the Theory of Point Processes: 
Volume I: Elementary Theory and Methods. }
Second Edition. 
Springer.

\bibitem{DVJ2}
Daley, D.J., Vere-Jones, D. (2008).
{\em An Introduction to the Theory of Point Processes: 
Volume II: General Theory and Structure. }
Second Edition. 
Springer.

\bibitem{GabrielDiggle}
Gabriel, E., Diggle, P.J. (2009).
Second-order analysis of inhomogeneous spatio-temporal point process data. 
{\em Statistica Neerlandica} 63, 34--51.

\bibitem{GabrielDiggleRowlingson}
Gabriel, E., Rowlingson, B., Diggle, P.J. (2013). 
Stpp: An R Package for plotting, simulating and analysing spatio-temporal 
point patterns.
{\em Journal of Statistical Software} 53, 1--29. 

\bibitem{Handbook}
Gelfand, A., Diggle, P., Fuentes, M., Guttorp, P. (2010). 
{\em Handbook of Spatial Statistics.}
Taylor \& Francis. 

\bibitem{Halmos}
Halmos, P.R. (1974). 
{\em Measure Theory. }
Springer. 

\bibitem{Illian}
Illian, J., Penttinen, A., Stoyan, H., Stoyan, D. (2008). 
{\em Statistical Analysis and Modelling of Spatial Point Patterns.}
Wiley-Interscience. 

\bibitem{MCbook}
Lieshout, M.N.M.~van (2000). 
{\em Markov Point Processes and Their Applications.}
Imperial College Press/World Scientific.

\bibitem{MCJfunMPP}
Lieshout, M.N.M.~van (2006). 
A $J$-function for marked point patterns.
{\em Annals of the Institute of Statistical Mathematics} 58, 
235--259. 

\bibitem{MCJfun}
Lieshout, M.N.M.~van (2011). 
A $J$-function for inhomogeneous point processes. 
{\em Statistica Neerlandica} 65, 183--201. 

\bibitem{MCBaddeley}
Lieshout, M.N.M.~van, Baddeley, A.J. (1996). 
A nonparametric measure of spatial interaction in point patterns.
{\em Statistica Neerlandica} 50, 344--361. 

\bibitem{MollerDiaz}
M\o ller, J., Diaz-Avalos, C. (2010). 
Structured spatio-temporal shot-noise Cox point process models, 
with a view to modelling forest fires. 
{\em Scandinavian Journal of Statistics} 37, 2--25.

\bibitem{MollerGhorbani}
M\o ller, J., Ghorbani, M. (2010). 
Aspects of second-order analysis of structured inhomogeneous spatio-temporal 
point processes.
{\em Statistica Neerlandica} 66, 472--491.

\bibitem{MollerSyversveen}
M{\o}ller, J., Syversveen, A.R., Waagepetersen, R.P. (1998). 
Log Gaussian Cox processes. 
{\em Scandinavian Journal of Statistics} 25, 451--482.

\bibitem{MollerWaagepetersen}
M{\o}ller, J., Waagepetersen, R.P. (2003). 
{\em Statistical Inference and Simulation for Spatial Point Processes.}
Chapman and Hall/CRC.

\bibitem{MollerWaagepetersenArticle}
M{\o}ller, J., Waagepetersen, R.P. (2007). 
Modern statistics for spatial point processes. 
{\em Scandinavian Journal of Statistics} 34, 643--711.

\bibitem{Ogata}
Ogata, Y. (1998). 
Space-time point-process models for earthquake occurrences. 
{\em Annals of the Institute of Statistical Mathematics} 50, 379--402.

\bibitem{Pitt}
Pitt, L.D. (1982). 
Positively correlated normal variables are associated. 
{\em Annals of Probability} 10, 496--499.

\bibitem{Rath96}
Rathbun, S.L. (1996).
Estimation of Poisson intensity using partially observed concomitant
variables.
{\em Biometrics} 52, 226--242.

\bibitem{Schlather}
Schlather, M.,  Menck, P., Singleton, R., Pfaff, B. (2013).
RandomFields: Simulation and analysis of random fields. \\
{\tt http://cran.r-project.org/web/packages/RandomFields/index.html}.

\bibitem{SchneiderWeil}
Schneider, R., Weil W. (2008). 
{\em Stochastic and Integral Geometry.}
Springer. 

\bibitem{Steenbeek}
Steenbeek, A.G., Lieshout, M.N.M.~van, Stoica, R.S. 
with contributions from Gregori, P. and Berthelsen, K.K. (2002--2003). 
MPPBLIB, a {\tt C++} library for marked point processes. 
CWI.

\bibitem{White}
White, S.D.M. (1979). 
The hierarchy of correlation functions and its relation to other 
measures of galaxy clustering. 
{\em Monthly Notices of the Royal Astronomical Society} 186, 145--154.

\end{thebibliography}
\end{document}